\documentclass[12pt]{amsart}
\usepackage{amssymb,amsmath,amscd,epsfig,amsthm,multicol,multirow}
\usepackage{dsfont}
\usepackage{hyperref}
\usepackage{eepic}

\allinethickness{0.4mm}

\newcommand{\Aut}{\mathop{\rm Aut}\nolimits}
\newcommand{\SHAut}{\mathop{\rm SHAut}\nolimits}

\newcommand{\Stg}{\mathop{\rm Stab}\nolimits_\G}

\newcommand{\G}{\mathcal G}
\newcommand{\A}{\mathcal A}

\newcommand{\LL}{\mathcal L}
\newcommand{\Z}{\mathbb Z}

\newcommand{\Sym}{\mathop{\rm Sym}\nolimits}
\newcommand{\Rist}{\mathop{\rm Rist}\nolimits}
\newcommand{\St}{\mathop{\rm Stab}\nolimits}
\newcommand{\triv}{\mathop{\rm triv}\nolimits}
\newcommand{\Fix}{\mathop{\rm Fix}\nolimits}

\newtheorem{theorem}{Theorem}[section]
\newtheorem{corollary}[theorem]{Corollary}
\newtheorem{proposition}[theorem]{Proposition}
\newtheorem{lemma}[theorem]{Lemma}
\newtheorem{definition}[theorem]{Definition}
\newtheorem{question}{Question}

\begin{document}
\title[Groups acting essentially freely on the boundary of the tree]{Self-similar groups acting essentially freely on the boundary of the binary rooted tree}

\author[R.Grigorchuk]{Rostislav~Grigorchuk}
      \address{Department of Mathematics\\
               Texas A\&M University\\
               College Station, TX, 77843}
\thanks{
The first author was partially  supported by NSF grant DMS-1207699 and ERC starting
grant GA 257110 RaWG}
\email{\href{grigorch@math.tamu.edu}{grigorch@math.tamu.edu}
}

\author[D.Savchuk]{Dmytro~Savchuk}
      \address{Department of Mathematics and Statistics\\
               University of South Florida\\
               4202 E Fowler Ave\\
               Tampa, FL 33620-5700}
\email{\href{mailto:savchuk@usf.edu}{savchuk@usf.edu}}


\begin{abstract}
We study the class of groups generated by automata that act essentially freely on the boundary of a rooted tree.  In the process we establish and discuss some general tools for determining if a group belongs to this class, and explore the connections of this class to the classes of just-infinite and scale-invariant groups. Our main application is a complete classification of groups generated by 3-state automata over 2-letter alphabet that are in this class.
\end{abstract}

\maketitle

\section{Introduction}
Groups generated by Mealy type automata represent an important and interesting class of groups with connections to different branches of mathematics, such as dynamical systems (including symbolic dynamics and holomorphic dynamics), computer science,  topology and probability.  Groups from this class were used to solve such important problems in group theory as Milnor's problem on groups of intermediate  growth,  Day problem on non-elementary amenability~\cite{grigorch:degrees}, and a strong Atiyah conjecture on $L^2$-Betti  numbers~\cite{grigorch_lsz:atiyah}. For more details about this class of groups we refer the reader to survey papers~\cite{gns00:automata,bartholdi_s:automata_groups10}.

In the whole class of groups generated by automata, there is an important subclass of self-similar groups.  These are the groups generated by  initial Mealy type automata that are determined by all states of a non initial automaton. The natural characteristic of such groups, which we will call \emph{complexity},  is the pair $(m,n)$ of two integers,  $m\geq 2, n\geq 2$,  where $m$ is a number of states and $n$ is a cardinality of the alphabet. There are 6 groups of complexity $(2,2)$ and the ``largest'' (most complicated) of them is the lamplighter group $\mathcal{L}=(\Z/2\Z)\wr\Z$. It is shown in~\cite{bondarenko_gkmnss:full_clas32_short} and~\cite{muntyan:phd} that there is not more than 115 different (up to isomorphism) groups of complexity $(3,2)$, although the number of corresponding automata up to certain natural symmetry is $194$. Even though the complete characterization of $(3,2)$-groups is not achieved yet, a lot of information about these groups has been obtained.

Study of groups generated by automata with small number of states and small alphabet is a very reasonable project which can be justified by following examples. An observation made in~\cite{gns00:automata} and~\cite{grigorch_z:lamplighter} that the lamplighter group can be generated by a 2-state automaton over a binary alphabet led to showing that the discrete Laplace operator on the Cayley graph of this group constructed using a generating set corresponding to states of automaton has pure point spectrum. This happened to be not only the first example of a group with discrete spectrum, but led to the construction of a counterexample to the strong Atiayh conjecture. Further, a careful search for an interesting (3,2)-groups allowed {\.Z}uk and the first author to bring the attention to a group generated by the automaton given in Figure~\ref{fig:basilica_autom} (see~\cite{grigorch_z:basilica}), that later got the name Basilica. With the help of this group not only one important problem on amenability was solved in~\cite{bartholdi_v:amenab} but an important method of proving amenability (now called the Munchausen trick) was developed. Moreover, study of this group initiated a new direction in holomorphic dynamics -- Iterated Monodromy Groups defined and studied by Nekrashevych and other researches~\cite{nekrash:self-similar}.
\begin{figure}
\begin{center}
\includegraphics{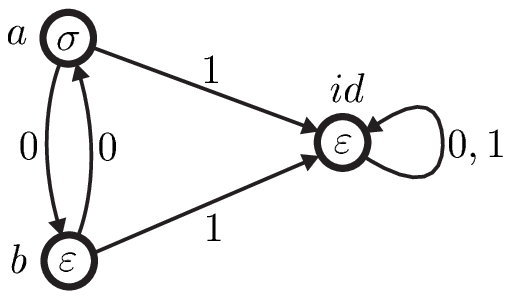}
\end{center}
\caption{Automaton generating an amenable but not subexponentially amenable Basilica group\label{fig:basilica_autom}}
\end{figure}

A principal discovery of the first author was that the class of (5,2)-automata groups contains groups of intermediate growth (between polynomial and exponential)~\cite{grigorch:milnor,grigorch:degrees}. Later Bux and P{\'e}rez in~\cite{bux_p:iter_monodromy} showed that such groups exist even among groups generated by (4,2)-automata.

Who knows what other problems, or interesting properties and directions of studies may come from careful study of groups of small complexity $(m,n)$? The authors are confident that approach based on careful study of (3,2), (2,3) and (4,2)-groups is perspective and productive. It also helps to understand what one can expect from the structure and properties of automaton groups, which in future may potentially lead to a result similar to Tits alternative.


Surprisingly, groups generated by automata are related to many topics in dynamical systems and ergodic theory. More generally, a far from being complete list of topics that have links to automata groups includes: fractal dynamics, symbolic dynamics, automatically generated sequences, Lyapunov stability, adding machines, etc. One of the links that we are going to exploit in this paper is as follows.

Groups generated by finite automata defined over the $m$-letter alphabet, in particular self-similar groups, naturally act on the $m$-regular rooted tree $T=T_m$ ($m$ being a cardinality of the alphabet) and on its boundary, which topologically is homeomorphic to the Cantor set.  This action preserves the uniform Bernoulli measure $\mu$ on the boundary.  Therefore, one can study a topological dynamical system  $(G,\partial T)$ or metric dynamical system $(G,\partial T, \mu)$. Ergodicity of the latter is equivalent to the level transitivity of the action of $G$ on $T$.

The important classes of group actions are topologically free actions and essentially free actions. For the first case, the assumption is that for each nonidentity element $g\in G$ the set of fixed points $Fix(g)$ is meager (i.e.~can be represented as a countable union of nowhere dense sets). In the second case we require that for any nonidentity element $g$ of a group the measure of the fixed point set of $g$ is zero.  These types of actions play especially important role in various studies in dynamical systems, operator algebras, and modern directions of group theory like theory of cost or rank gradient \cite{gaboriau:what_is_cost10,abert_n:rank_gradient12}. Self-similar groups acting essentially freely on $\partial T$ can potentially be used to construct new examples of scale-invariant groups~\cite{nekrashevych_p:scale_invariant} (we call a group $G$ scale-invariant if there is a sequence of finite index subgroups of $G$ that are all isomorphic to $G$ and whose intersection is trivial), and have connection to the class of hereditary just-infinite groups~\cite{grigorch:jibranch}. They also may lead   to the construction of new examples of expenders as indicated in~\cite{grigorch:dynamics11eng}. In the situation of a randomly chosen group acting on (unrooted) tree, typical actions are essentially free (see, for example,~\cite{abert_g:free_actions09}). But in the situation we consider in this paper, the freeness of the action is a rather rear event that sometimes requires nontrivial proofs.

The opposite to free actions are totally non-free actions considered recently in~\cite{grigorch:dynamics11eng,vershik:nonfree_sym_infty}. These are the actions, for which stabilizers of different points of the set of full measure are different. Surprisingly many groups generated by finite automata, in particular those of them that are branch or weakly branch, act totally non-free.  Totally non free actions are also important for the theory of operator algebras and for rapidly developing now theory of invariant random subgroups \cite{vershik:nonfree_sym_infty,abert_gv:kesten_irs12,bowen:irs_free12,bowen_gk:irs_lamplighter,dudko_m:higman_thompson12,dudko_m:characters_of_inductive_limits13}.

The goal of this paper is to describe all $(3,2)$-groups acting essentially freely on the boundary.  Although in general, for group actions on topological spaces with invariant measure, there is no connection between topological freeness and essential freeness, in the case of groups generated by finite automata acting on the boundary of a tree (in a way prescribed by determining automaton) these two notions are equivalent, as observed by Kambites, Silva and Steinberg in~\cite{kambites-s-s:spectra}.

To each (3,2)-automaton one assigns a unique number from 1 to 5832 according to certain natural lexicographic order on the set of all these automata (see Section~\ref{sec:autom} and~\cite{bondarenko_gkmnss:full_clas32_short}). Obviously, two automata whose minimizations can be obtained from each other by permuting the states, letters, or passing to the inverse automaton, generated isomorphic groups whose actions on $\partial T_2$ are conjugate. This defines an equivalence relation on the set of all automata that we call \emph{minimal symmetry} (this term reflects the fact that we first minimize automata before looking for a symmetry). By definition of this relation, up to group isomorphism for each equivalence class it is enough to study only one representative. Moreover, the action of a group generated by automaton $\A$ on the boundary of the tree is essentially free if and only if the action of a group generated by any automaton minimally symmetric to $\A$ has this property. In the main theorem below we list all groups generated by (3,2)-automata acting essentially freely on $\partial T_2$ and for each group we give in brackets the numbers of representatives of equivalence classes of automata that generate this group.

Or main result is:
\begin{theorem}
\label{thm:main}
Among all groups generated by $3$-state automata over 2-letter alphabet the only groups that act essentially freely on the boundary of the tree $T_2$ are the following:
\begin{itemize}
\item Trivial group [1];
\item Group $\Z/2\Z$ of order 2 [1090,1094];
\item Klein group $(\Z/2\Z)\times(\Z/2\Z)=\langle a,b\mid a^b=a^{-1}\rangle$\\~ [730,734,766,770,774,2232,2264,2844,2880];
\item $(\Z/2\Z)\times(\Z/2\Z)\times(\Z/2\Z)$ [802,806,810,2196,2260];
\item Infinite cyclic group $\Z$ [731,767,768,804,1091,2861,2887];
\item $\Z^2$ [771,803,807];
\item Infinite dihedral group $D_\infty$ [820,824,865,919,928,932,936,2226,2358,2394,\\2422,2874];
\item Baumslag-Solitar group $BS(1,3)=\langle t,x\mid t^x=t^3\rangle$ [870,924];
\item Baumslag-Solitar group $BS(1,-3)=\langle t,x\mid t^x=t^{-3}\rangle$ [2294,2320];
\item Extension $\bigl((\Z/2\Z)\wr Z\bigr)\rtimes(\Z/2\Z)$ of the lamplighter group by $\Z/2\Z$, where the nontrivial element of $\Z/2\Z$ inverts the canonical generators of the lamplighter group [891];
\item Free group $F_3$ of rank 3 generated by the Aleshin automaton [2240];
\item Free product $(\Z/2\Z)*(\Z/2\Z)*(\Z/2\Z)$ of three groups of order 2 generated by Bellaterra automaton [846];
\item Lamplighter group $\LL\cong(\Z/2\Z)\wr\Z$ [821,839,930,2374,2388];
\item Extension $\Z^2\rtimes(\Z/2\Z)$ of the $\Z^2$ by $\Z/2\Z$, where the nontrivial element of $\Z/2\Z$ inverts the elements of $\Z^2$ [2277,2313,2426];
\item Extension $\bigl((\Z/2\Z)^2\wr \Z\bigr)\rtimes (\Z/2\Z)$ of a rank 2 lamplighter group $\mathcal L_{2,2}\cong (\Z/2\Z)^2\wr \Z$ by $\Z/2\Z$, where the action of $\Z/2\Z$ on $\mathcal L_{2,2}$ is described in Theorem~\ref{thm:structure2193} [2193];
\item Extension $BS(1,3)\rtimes (\Z/2\Z)$ of Baumslag-Solitar group $BS(1,3)$ by $\Z/2\Z$, where the generator of $\Z/2\Z$ acts on $BS(1,3)=\langle t,x\mid t^x=t^3\rangle$ by inverting $t$ and fixing $x$ [2372],
\end{itemize}
where the numbers in brackets indicate corresponding numbers of $(3,2)$-automata defined in Section~\ref{sec:autom}. Moreover, all groups in this list except finite nontrivial groups, $F_3$, and $(\Z/2\Z)*(\Z/2\Z)*(\Z/2\Z)$ are scale-invariant.
\end{theorem}

Note that the notation $\LL_{2,2}$ used in the above theorem is borrowed from~\cite{grigorch_k:lamplighter}, where $\LL_{p,n}$ denotes the group $(\Z/p\Z)^n\wr\Z$ called the rank $n$ lamplighter group. We also denote throughout the paper by $\LL$ the ``standard'' lamplighter group $\LL_{2,1}$. Also, throughout the paper $BS(1,n)$ will denote the Baumslag-Solitar group isomorphic to $\langle t,x\mid t^x=t^n\rangle$.

The paper is organized as follows. In Section~\ref{sec:autom} we recall main definitions from a theory of groups generated by automata, and introduce necessary notation related to the class of 3-state automata over 2-letter alphabet. Section~\ref{sec:types_of_actions} discusses various types of free actions and lists relevant results in this area. The main Theorem~\ref{thm:main} is proved in Section~\ref{sec:proof}. Finally, we finish the paper with open questions and concluding remarks in Section~\ref{sec:concluding}.

\noindent{\bf Acknowledgement.} The authors are sincerely grateful to Tatiana Smirnova-Nagnibeda and Volodymyr Nekrashevych for valuable comments, suggestions and discussions that helped to improve the paper. We also would like to thank the anonymous referee for numerous suggestions that enhanced the paper and significantly simplified the arguments in Subsection~\ref{ssec:2372}.

\section{Groups generated by automata and classification notations}
\label{sec:autom}
In this section we remind the main notions related to automaton groups and to the problem of classification of (3,2)-groups.

Let $X$ be a finite set of cardinality $d$ and let $X^*$ denote the the set of all finite words over $X$. This set naturally serves as a vertex set of a rooted tree in which vertex $v$ is adjacent to $vx$ for
any $v\in X^*$ and $x\in X$. The empty word $\emptyset$ corresponds to the root
of the tree and $X^n$ corresponds to the $n$-th level of the tree. Sometimes we will denote this tree by $T(X)$ or by $T_d$ if $|X|=d$.
We will be interested in the groups of automorphisms and semigroups
of endomorphisms of the tree $X^*$, where by endomorphisms we mean maps from the set of vertices $V(X^*)$ to itself that preserve the root $\emptyset$ and adjacency relation, and by automorphisms of $X^*$ we mean bijective endomorphisms. Any endomorphism of $X^*$ can be defined via
the notion of an initial automaton as described below.

\begin{definition}
A \emph{Mealy automaton} (or simply \emph{automaton}) is a tuple
$(Q,X,\pi,\lambda)$, where $Q$ is a set of states, $X$ is a
finite alphabet, $\pi\colon Q\times X\to Q$ is a transition function
and $\lambda\colon Q\times X\to X$ is an output function. If the set
of states $Q$ is finite the automaton is called \emph{finite}. If
for every state $q\in Q$ the output function $\lambda(q,x)$ induces
a permutation of $X$, the automaton $\A$ is called \emph{invertible}.
Selecting a state $q\in Q$ produces an \emph{initial automaton}
$\A_q$.
\end{definition}

Automata are often represented by the \emph{Moore diagrams}. The
Moore diagram of an automaton $\A=(Q,X,\pi,\lambda)$ is a directed
graph in which the vertices are the elements of $Q$ and the edges
have form $q\stackrel{x|\lambda(q,x)}{\longrightarrow}\pi(q,x)$ for
$q\in Q$ and $x\in X$. If the automaton is invertible, then it is
convenient to label vertices of the Moore diagram by the permutation
$\lambda(q,\cdot)$ and leave just first components from the labels
of the edges. To distinguish these two ways to draw a Moore diagram we will call the former type by Moore diagram of type I and the latter one by Moore diagram of type II. An example of Moore diagram of type II is shown in Figure~\ref{fig:L22autom}.

Any initial automaton induces a homomorphism of $X^*$. Given a word
$v=x_1x_2x_3\ldots x_n\in X^*$ it scans its first letter $x_1$ and
outputs $\lambda(x_1)$. The rest of the word is handled in a similar
fashion by the initial automaton $\A_{\pi(x_1)}$. In other words,
the functions $\pi$ and $\lambda$ can be extended to $\pi\colon
Q\times X^*\to Q$ and $\lambda\colon  Q\times X^*\to X^*$ via
\[\begin{array}{l}
\pi(q,x_1x_2\ldots x_n)=\pi(\pi(q,x_1),x_2x_3\ldots x_n),\\
\lambda(q,x_1x_2\ldots x_n)=\lambda(q,x_1)\lambda(\pi(q,x_1),x_2x_3\ldots x_n).\\
\end{array}
\]

By construction, any initial automaton acts on $X^*$ (viewed as a tree) as an
endomorphism. In the case of invertible automaton it acts as an
automorphism.

\begin{definition}
The semigroup (group) generated by all states of automaton $\A$ is
called the \emph{automaton semigroup} (\emph{automaton group}) and
denoted by $\mathds S(\A)$ (respectively $\mathds G(\A)$).
\end{definition}

Note, that the composition and the inverse of transformations defined by (finite) automata are again defined by (finite) automata. For example, the \emph{inverse automaton} $\A_q^{-1}$ to automaton $\A_q$ defining the inverse of the transformation of $X^*$ defined by $\A_q$ is obtained from $\A_q$ simply by flipping the components of the labels of all edges in its Moore diagram of type I.

Among general properties of automaton groups we will use that all of them are residually finite,  and thus Hopfian by Malcev's theorem~\cite{malcev:hopfian40}, i.e. each surjective endomorphism of an automaton group on itself is an isomorphism.

We will need a notion of a
\emph{section} of a homomorphism at a vertex of the tree. Let $g$ be
a homomorphism of the tree $X^*$ and $x\in X$. Then for any $v\in
X^*$ we have
\[g(xv)=g(x)v'\]
for some $v'\in X^*$. The map $g|_x\colon X^*\to X^*$ given by
\[g|_x(v)=v'\]
defines an endomorphism of $X^*$ and is called the \emph{section} of
$g$ at vertex $x$. Furthermore,  for any $x_1x_2\ldots x_n\in X^*$
we define \[g|_{x_1x_2\ldots x_n}=g|_{x_1}|_{x_2}\ldots|_{x_n}.\]

Given an endomorphism $g$ of $X^*$ one can construct an initial automaton
$\A(g)$ whose action on $X^*$ coincides with that of $g$ as follows.
The set of states of $\A(g)$ is the set $\{g|_v\colon  v\in X^*\}$
of different sections of $g$ at the vertices of the tree. The
transition and output functions are defined by
\[\begin{array}{l}
\pi(g|_v,x)=g|_{vx},\\
\lambda(g|_v,x)=g|_v(x).
\end{array}\]

Throughout the paper we will use the following convention. If $g$
and $h$ are the elements of some (semi)group acting on set $A$ and
$a\in A$, then
\begin{equation}
\label{eqn_conv}
gh(a)=h(g(a)).
\end{equation}
In particular, this means that we consider right action of $\Sym(X)$ on $X$. This agrees with the order of multiplication of permutations in \verb"GAP" (also corresponding to the right action) that we use extensively below. But for convenience of further notation we will still write the elements of the group on left. The reason we do not use the right action written on right lies in the standard convention to write words over finite alphabet from left to right, which means that when an element of an automaton (semi)group $g$ acts on a word $x_1x_2\ldots x_n$, it first processes the leftmost letter, then the second from left, etc.

Taking into account convention~\eqref{eqn_conv} one can compute
sections of any element of an automaton semigroup as follows. If
$g=g_1g_2\cdots g_n$ and $v\in X^*$, then

\begin{equation}
\label{eqn_sections} g|_v=g_1|_v\cdot g_2|_{g_1(v)}\cdots
g_n|_{g_1g_2\cdots g_{n-1}(v)}.
\end{equation}

Another popular name for automaton groups and semigroups is
self-similar groups and semigroups
(see~\cite{nekrash:self-similar}).

\begin{definition}
A (semi)group $G$ of (homomorphisms) automorphisms of $X^*$ is called \emph{self-similar} if all sections of each element of $G$ belong to $G$.
\end{definition}

Clearly every automaton group is self-similar as the sections of the generator of every such group are again generators, and for other elements it follows from the fact that the sections of the product are computed as products of sections. On the other hand, each self-similar group $G$ is generated by all states of automata corresponding to all of its elements. The union of all these automata is the automaton generating $G$ (possibly not the smallest one). In the rest of the paper depending on the context we will use both these terms.

Self-similarity allows us to define a natural embedding of any automaton group $G$
\[G\hookrightarrow G \wr \Sym(X)\]
defined by
\begin{equation}
\label{eqn:wreath}
G\ni g\mapsto (g|_0,g|_1,\ldots,g|_{d-1})\lambda(g)\in G\wr \Sym(X),
\end{equation}
where $g|_0,g|_1,\ldots,g|_{d-1}$ are the sections of $g$ at the vertices
of the first level, and $\lambda(g)$ is a permutation of $X$ induced by the action of $g$ on the first level of the tree.

The above embedding is convenient in computations involving the
sections of automorphisms, as well as for defining automaton groups. We will call it the \emph{wreath recursion} defining the group.

The following important notions related to groups generated by automata will be used throughout the text.

\begin{definition}
A self-similar group $G$ is \emph{self-replicating} if, for every vertex $u\in X^*$, the homomorphism $\phi_u\colon \St_G(u)\to G$ from the stabilizer of the vertex $u$ in $G$ to $G$, given by $\phi_u(g) = g|_u$, is surjective.
\end{definition}

\begin{definition}
We say that an element $g$ of a self-similar group (resp., a self-similar group $G$) acts spherically transitively, if $g$ (resp., $G$) acts transitively on each level $X^n$ of the tree $X^*$.
\end{definition}

Note, that a self-similar groups acting on binary tree is infinite if and only if it acts spherically transitively (see Lemma~3 in~\cite{bondarenko_gkmnss:full_clas32_short}).

An important class of groups acting on trees is the class of branch groups~\cite{grigorch:jibranch,bar_gs:branch}.

\begin{definition}Let $G$ be a group acting on the rooted tree $X^*$.\\[-3mm]
\label{def:rigid}
\begin{itemize}
\item The \emph{rigid stabilizer of a vertex $v\in X^*$} in $G$ is a subgroup $\Rist_G(v)$ of $G$ that consists of elements that act nontrivially only on the vertices that have $v$ as a prefix.\\
\item The \emph{rigid stabilizer of a level $n$} of $X^*$ in $G$ is a subgroup $\Rist_G(n)$ of $G$ that is generated by rigid stabilizers of all the vertices of this level.
\end{itemize}
\end{definition}

\begin{definition}
A group $G$ acting on the rooted tree $X^*$ is called\\[-3mm]
\begin{itemize}
\item \emph{weakly branch}, if for each $n\ge1$ the rigid stabilizer $\Rist_n(G)$ of the $n$-th level of $X^*$ is nontrivial;
\item \emph{branch}, if for each $n\ge1$ the rigid stabilizer $\Rist_n(G)$ of the $n$-th level of $X^*$ has finite index in $G$.
\end{itemize}
\end{definition}

Further, we will need a notion of a dual automaton $\hat\A$ to automaton $\A$, which is obtained from $\A$ by ``switching the roles'' of states and letters of the alphabet. The formal definition is given below.

\begin{definition}
Given a finite automaton $\A=(Q,X,\pi,\lambda)$ its \emph{dual
automaton} $\hat\A$ is a finite automaton
$(X,Q,\hat\lambda,\hat\pi)$, where
\[\begin{array}{l}
\hat\lambda(x,q)=\lambda(q,x),\\
\hat\pi(x,q)=\pi(q,x)
\end{array}\]
for any $x\in X$ and $q\in Q$.
\end{definition}

Note that the dual of the dual of an automaton $\A$ coincides with
$\A$. The semigroup $\mathds S(\hat\A)$ generated by dual automaton
$\hat\A$ acts on $Q^*$. This
action induces the action on $\mathds S(\A)$. Similarly, $\mathds
S(\A)$ acts on $\mathds S(\hat\A)$.

\begin{definition}
For an automaton semigroup $G$ generated by automaton $\A$ the
\emph{dual semigroup} $\hat G$ to $G$ is a semigroup generated by a
dual automaton $\hat\A$.
\end{definition}

A particularly important class of automata is the class of
bireversible automata as they give rise to interesting examples of groups, provide an approach to prove freeness properties, and admit solutions to certain algorithmic problems~\cite{gl_mo:compl,savchuk_v:free_prods,akhavi_klmp:finiteness_problem,klimann:finiteness}.

\begin{definition}
An automaton $\A$ is called \emph{bireversible} if it is invertible,
its dual is invertible, and the dual to $\A^{-1}$ are invertible.
\end{definition}

Now we describe shortly the notation and some basic facts used in the classification of $(3,2)$-groups~\cite{bondarenko_gkmnss:full_clas32_short}. These groups act on a binary rooted tree $T_2=X^*$ for $X=\{0,1\}$ and throughout the rest of the paper we will denote by $\sigma=(01)$ the nontrivial permutation of letters in $X$. We will usually omit writing the trivial permutation in wreath recursions, but sometimes we denote it by $\sigma^0$.

To every invertible 3-state automaton $\A$ with set of states
$S=\{\textbf0,\textbf1,\textbf2\}$ acting on the 2-letter alphabet
$X$ we assign a unique number as follows. Given the wreath
recursion
\[
 \left\{
\begin{array}{l}
  \textbf0=(a_{11},a_{12})\sigma^{a_{13}}, \\
  \textbf1=(a_{21},a_{22})\sigma^{a_{23}},  \\
  \textbf2=(a_{31},a_{32})\sigma^{a_{33}},
\end{array}
\right.
\]
defining the automaton $\A$, where $a_{ij}\in
\{\textbf0,\textbf1,\textbf2\}$ for $j\neq1$ and $a_{i3}\in
\{0,1\}$, $i=1,2,3$, assign the number
\[
\begin{array}{l}
\mathop{\rm Number}\nolimits(\mathcal A)= \\
\qquad\qquad a_{11}+3a_{12}+9a_{21}+27a_{22}+81a_{31}+\\
 \qquad\qquad 243a_{32}+729(a_{13}+2a_{23}+4a_{33})+1
\end{array}
\]
to $\A$. With this agreement the numbers assigned to automata range from $1$ to $5832$. The numbering of the automata is induced by the lexicographic ordering of tuples $(a_{11},a_{12},\ldots,a_{33})$ that define all automata in
the class. Each of the automata numbered $1$ through $729$ generates
the trivial group, since all vertex permutations are trivial in this
case. Each of the automata numbered $5104$ through $5832$ generates
the cyclic group $\Z/2\Z$ of order $2$, since both states represent the
automorphism that acts by changing all letters in every word over
$X$. Therefore the nontrivial part of the classification is
concerned with the automata numbered by $730$ through $5103$.

Denote by $\mathcal A_n$ the automaton numbered by $n$ and by $G_n$
the corresponding group $\mathds G(\A_n)$ of tree automorphisms. Sometimes we will use
just the number to refer to the corresponding automaton or group.

The following three operations on automata do not change the
isomorphism class of the group generated by the corresponding
automaton (and do not change the action on the tree up to conjugation):
\begin{enumerate}
\item[(i)] passing to inverses of all generators (equivalently, passing to the inverse automaton),
\item[(ii)] permuting the states of the automaton,
\item[(iii)] permuting the alphabet letters.
\end{enumerate}

\begin{definition}
Two automata $\mathcal A$ and $\mathcal B$ that can be obtained from
one another by using a composition of the operations ($i$)--($iii$),
are called \emph{symmetric}.
\end{definition}

Additional identifications can be made after automata minimization
is applied. Recall, that the minimization of an automaton is a standard procedure (see, for example,~\cite{eilenberg:automata2}) that identifies the states that induce identical transformations of $X^*$.

\begin{definition}
\label{def:minim_sym}
If the minimization of an automaton $\mathcal A$ is symmetric to the
minimization of an automaton $\mathcal B$, we say that the automata
$\mathcal A$ and $\mathcal B$ are \emph{minimally symmetric} and
write $\mathcal A\sim\mathcal B$.
\end{definition}

There are $194$ classes of $(3,2)$-automata that are pairwise not
minimally symmetric. At present, it is known that there are no more than $115$
non-isomorphic $(3,2)$-automaton groups and all these groups are listed in~\cite{bondarenko_gkmnss:full_clas32_short,muntyan:phd}.

In this paper, since we are looking for essentially free actions of groups, we will actually distinguish non minimally symmetric automata generating isomorphic groups, as the same group may have different actions on $\partial T_2$. So we will work with all 194 classes of not minimally symmetric automata.

\section{Types of actions and main tools}
\label{sec:types_of_actions}

There are different ways to define the freeness of a group actions. The definition below works in the general context of arbitrary topological (or, respectively, measure) space, but we will work only in the context of the actions of self-similar groups on the boundary $\partial T$ of the rooted tree $T$. Recall, that $\partial T$ consists of all infinite paths without backtracking initiating from the root (equivalently, $\partial X^*$ can be thought of as the set of all infinite words over $X$). The set $\partial T$ is endowed with a topology in which two paths are declared to be close if they have long common beginning. With this topology it is homeomorphic to the Cantor set. Further, one can define a uniform Bernoulli measure on $\partial T$ making this space a measure space. This measure is invariant under the action of the group of all automorphisms of $T$. Moreover, for any group $G<\Aut(T)$ acting spherically transitively on the levels of $T$, the uniform Bernoulli measure is a unique $\sigma$-additive $G$-invariant probabilistic measure on $\partial T$ (see Proposition 6.5 in~\cite{gns00:automata}).

Now we remind the general definition and set up some notation. Let $G$ be a countable group acting on a complete metric space $Y$. Denote by $Y_-$ the set of points with nontrivial stabilizer and by $Y_+$ the set of points with trivial stabilizer.

\begin{definition}~\\[-3mm]
\begin{enumerate}
\item The action $(G,Y)$ is said to be \emph{absolutely free} if all points have trivial stabilizers.
\item The action $(G,Y)$ is \emph{topologically free} if $Y_-$ is a meager set (i.e., it can be represented as a countable union of nowhere dense sets).
\item Suppose that the action $(G,Y)$ has a $G$-invariant (not necessarily finite) Borel measure $\mu$.
The action on the measure space $(G,Y,\mu)$ is said to be \emph{essentially free} if $\mu(Y_-)=0$.
\end{enumerate}
\end{definition}

In the context of self-similar groups acting on the boundary $\partial T(X)$ of corresponding tree, this gives immediately topological dynamical system $(G,\partial T(X))$. As mentioned above, $\partial T(X)$ can be considered as a measure space with a uniform Bernoulli measure, which enables us to talk about the essential freeness of the action of $G$ on $\partial T(X)$. An important result here is that in the case of groups generated by finite state automata the notions of topological freeness and essential freeness are identical according to the following two propositions.

\begin{proposition}[\cite{grigorch:dynamics11eng}, Corollary 4.3]
A spherically transitive essentially free action on the boundary of a tree is topologically free.
\end{proposition}

\begin{theorem}[\cite{kambites-s-s:spectra}, Theorem 4.2.]
\label{thm:freeness_equiv}
For groups generated by finite automata, any topologically free action is essentially free.
\end{theorem}

We note, that the terminology used in~\cite{kambites-s-s:spectra,steinberg_vv:series_free} is somewhat different from the one used here. For example, the topological freeness bears the name of freeness in the sense of Baire category, and the essential freeness is referred to as freeness in the sense of ergodic theory. Further, the definitions used for these types of freeness are different, but equivalent in the case of countable groups (in which we are interested anyway). Namely, if for $g\in G$ one denotes by $\Fix(g)$ the subset of $X$ fixed by $g$, then we have
\[X_-=\cup_{g\in G}\Fix(g).\]
Therefore, if $G$ is countable, one can replace the condition that $X_-$ has measure zero (resp., $X_-$ is meager) by the equivalent condition that $\Fix(g)$ has measure zero (resp., meager) for each nonidentity $g\in G$.

In order to establish that a group does not act topologically (and essentially) freely, one can just find an element $g\in G$ and a vertex $v\in X^*$ fixed by $g$ such that $g|_v$ is trivial (because in this case all points in the cylindrical set $c_v$, which is open (and has positive measure) will have $g$ in their stabilizers.

\begin{definition}
For a vertex $v\in X^*$ the set of all $g\in G$ that fix $v$ and such that $g|_v$ is trivial forms a subgroup $\triv_G(v)$ of $G$ called the \emph{trivializer} of $v$.
\end{definition}

\begin{definition}
The action of a group $G$ on a rooted tree is called \emph{locally nontrivial} if trivializers of all vertices of the tree are trivial.
\end{definition}

As observed above, if the action is not locally trivial, it cannot be topologically or essentially free. It is not hard to prove the converse in the case of countable group and topological freeness.

\begin{proposition}[\cite{grigorch:dynamics11eng}, Proposition 4.2.]
\label{prop:loc_triv}
The action of a countable group on the boundary of a tree is topologically free if and only if it is locally nontrivial.
\end{proposition}

This observation, together with Theorem~\ref{thm:freeness_equiv}, constitutes one of the main tools to determine that a self-similar group \emph{does not} act essentially freely on the boundary of a tree. Of course, one can simply apply a brute force to find such an element, but in case of self-replicating groups it can be made almost automatic in many cases by using the the procedure that we describe below. This procedure is based on ideas of Mikhailova~\cite{mihailova:direct_products58} and is outlined in Section 5 of~\cite{grigorch:dynamics11eng}.

Suppose $G=\mathds G(\A)$ is a group generated by automaton $\A$ with states $a_1,a_2,\ldots,a_n$. With a slight abuse of notation we will treat $a_i$'s as generators of $G$ and write $G=\langle a_1,a_2,\ldots,a_n\rangle$. First, we calculate the finite generating set $\{s_j,\ j\in J\}$ of the stabilizer of the first level of the tree $\St_G(1)$ in $G$. This is a subgroup of finite index and a Reidemeister-Schreier procedure can be used for that.

Let $F_A$ denote the free group generated by elements $a_1,a_2,...,a_n$. The wreath recursion that defines an automaton induces an embedding
\[F_A\hookrightarrow F_A \wr \Sym(X)\]
defined by
\begin{equation}
\label{eqn:wreath_free}
F_A\ni g\mapsto (g|_0,g|_1,\ldots,g|_{d-1})\lambda(g)\in F_A\wr \Sym(X).
\end{equation}
With a slight abuse of notation, we will denote by $s_j$ also a word over $A\cup A^{-1}$ in $F_A$ that is mapped to $s_j\in G$ under the canonical epimorphism $F_A\to G$. Then we decompose each $s_j\in F_A$ as a pair $(s_j|_0, s_j|_1)\in F_A\times F_A$ using the wreath recursion embedding~\eqref{eqn:wreath_free}. The first components $s_j|_0$ of above pairs generate a subgroup $H$ of $F_A$. After applying the Nielsen reduction to the generators of this subgroup, keeping track of second coordinates, we obtain the generating set of $\langle (s_j|_0, s_j|_1),\ j\in J\rangle<F_A\times F_A$ whose projection onto the first coordinate is Nielsen reduced~\cite{lyndon_s:comb_grp_theory}:
\begin{equation}
\label{eqn:mikh1}
t_1=(b_1,w_1),\ldots,t_l=(b_m,w_m),\quad t_{m+1}=(1,r_{1}),\ldots,t_{m+l}=(1,r_{l}),
\end{equation}
where $\{b_1,\ldots,b_m\}$ is a Nielsen reduced generating set for $H$, $w_i\in F_A$ and $m+l=|J|$. We will call such a representation for $\St_G(1)$ the \emph{Mikhailova system} for $G$. The reason for such name is explained below.

If any of $r_i$, $i=1,\ldots,l$ represents a nonidentity element of $G$, then the corresponding pair $(1,r_i)$ will represent a nonidentity element of $G$ that belongs to the trivializer of vertex $1$. Thus, the action of $G$ on $\partial T_2$ would not be essentially free.

Showing that the group actually does act essentially freely is usually much harder, as witnessed by the last two sections. The main tool here is the Proposition~\ref{prop:rist} below. This proposition is similar to Proposition~\ref{prop:loc_triv}, but it additionally uses self-similarity of a group. Recall that the notion of a rigid stabilizer was introduced in Definition~\ref{def:rigid}.

\begin{proposition}[\cite{grigorch:dynamics11eng}, Proposition 4.5.]
\label{prop:rist}
For a group $G$ generated by finite automaton, acting on a binary tree $T_2$, the action on $\partial T_2$ is essentially free if and only if the rigid stabilizer of the first level $\Rist_{G}(1)$ is trivial.
\end{proposition}

The problem is that it is usually harder to show that the rigid stabilizer is trivial, than to find an element witnessing its nontriviality. The main method here is based on finding the presentation of a group by generators and relators. Note, that for a non-binary tree the condition of local nontriviality cannot be formulated in terms of rigid stabilizers.

We now go back to Equations~\eqref{eqn:mikh1}. In the case when $H$ coincides with $F_A$, which is the case when $G$ is self-replicating, we get $m=n$ and this equation is transformed to (after reordering the generators, if necessary):

\[t_1=(a_1,w_1),\ldots,t_l=(a_n,w_n),\quad t_{n+1}=(1,r_{1}),\ldots,t_{n+l}=(1,r_{l}).
\]

We can further assume that all $r_i$'s represent the identity element in $G$ (otherwise, as stated above, the action of $G$ is not essentially free). Suppose additionally that
\[\langle w_1,w_2,\ldots,w_n\rangle=F_A.\]
Then the map $\phi\colon a_i\to w_i$ extends to an automorphism of $F_A$. In this case we say that the presentation of the group $G$ by a finite automaton belongs to the \emph{diagonal type}. This condition does not depend on how the pairs of elements are reduced by the
Nielsen transformations. Note, that the case when $\phi$ is the identity automorphism, one obtains a subgroup of $F_A\times F_A$ that was used by Mikhailova in~\cite{mihailova:direct_products58} to prove that the membership problem for direct products of free groups is algorithmically unsolvable. This is why we attribute this notion to Mikhailova.

The following proposition follows immediately from Proposition~\ref{prop:rist}.

\begin{proposition}[\cite{grigorch:dynamics11eng}, Proposition 5.1]
\label{prop:ess_free_aut}
Suppose that $G$ is a group generated by finite automaton acting on a binary tree and
having the first-level stabilizer that can be reduced by the Nielsen transformations to the diagonal type.
Let $\phi$ be the above-constructed automorphism of the free group $F_A$. Then the action is essentially free if and only if $\phi$ induces an automorphism of the group $G$.
\end{proposition}

For some groups we use the following useful proposition that allows us to establish essential freeness of the action in the case of groups generated by finite bireversible automata, i.e. invertible automata, whose dual, and dual to the inverse are invertible as well.

\begin{proposition}[\cite{steinberg_vv:series_free}, Corollary 2.10]
\label{prop:bireversible}
A group generated by a bireversible automaton acts topologically and essentially freely on the boundary of the tree.
\end{proposition}

In the end of this section we would like to bring the attention to the connection between groups acting essentially freely on $\partial T_2$ and other classes of groups. Namely, we prove that each hereditary just-infinite self-similar group acts essentially freely on $\partial T_d$, and that the essentially free groups could be used to create new examples of scale-invariant groups. We start from definitions.

\begin{definition}
A group $G$ is called \emph{just-infinite} if it is infinite, but each proper quotient of $G$ is finite.
\end{definition}

\begin{definition}[\cite{grigorch:jibranch}]
A residually-finite group $G$ is called \emph{hereditary just-infinite} if each finite index subgroup of $G$ is just-infinite.
\end{definition}

Note that both hereditary just-infinite groups and branch groups play a crucial role in the trichotomy classifying finitely generated just-infinite groups~\cite{grigorch:jibranch}. According to this trichotomy any finitely generated
just-infinite group is
either branch  just-infinite group,  or hereditary  just  infinite group, or near simple group  (i.e.  a just-infinite group containing a subgroup of finite index that is a direct product of finitely many copies of a simple group).

\begin{proposition}
\label{prop:heredji}
Each hereditary just-infinite self-similar group $G$ generated by (possibly infinite) automaton over alphabet $X$ that acts transitively on the first level of $X^*$, acts essentially freely on $\partial T_{|X|}$.
\end{proposition}

Note that in the case of binary tree ($|X|=2$) the condition of transitivity of $G$ on the first level of $X^*$ is satisfied automatically.

\begin{proof}
Suppose the action on $\partial T_{|X|}$ is not essentially free. Then by Proposition~\ref{prop:loc_triv} there is a nonidentity element $g$ and a vertex $v\in X^*$ fixed by $g$ with $g|_v=1$. By self-similarity, we may assume that $v$ is a vertex of the first level.

For each $w\in X^*$ let $T_w$ denote the tree hanging down from the vertex $w$ and let $M_w=\St_G(w)|_{T_w}$ be the group consisting of all sections of elements of $\St_G(w)$ at vertex $w$.
Since $G$ acts transitively on the first level of $X^*$, all groups $M_w$ for $w\in X^1$ are conjugate. In particular, they are either all finite or all infinite. On the other hand, if all of $M_w$, $w\in X^1$ are finite, then $\St_G(1)$ must be finite as it embeds into $\prod_{w\in X^1}M_w$ via
\[\St_G(1)\ni g\mapsto (g|_1,g|_2,\ldots,g|_{|X|})\in\prod_{w\in X^1}M_w.\]
Since $G$ is infinite and $\St_G(1)$ has finite index in $G$, we conclude that $M_v$ is infinite. Now consider an epimorphism
\[\psi\colon\St_G(1)\to M_v\]
defined by $\psi(g)=g|_v$. Since $M_v$ is infinite and $\psi$ is onto, the kernel of $\psi$ has an infinite index in $\St_G(1)$, contradicting to the fact that $\St_G(1)$ is just-infinite, which must be the case as $G$ is hereditary just-infinite and $\St_G(1)$ has a finite index in $G$.
\end{proof}

We note, however, that there is currently no known examples of hereditary just-infinite self-similar groups. In view of this the above proposition tells that we have to look for such examples in the class of groups that act essentially freely on the boundary of the tree (see Question~\ref{que:heredji} in Section~\ref{sec:concluding}).

Recall, that a group $G$ is called \emph{B-scale-invariant} if there is a sequence of finite index subgroups of $G$ that are all isomorphic to $G$ and whose intersection is a finite group. This class was introduced by Benjamini (this is why we add ``B'' in front of ``scale-invariant'') who conjectured that every such group is virtually nilpotent. A counterexample based on the lamplighter group was provided implicitly in~\cite{grigorch_z:lamplighter} (the paper was printed before the conjecture was stated) and explicitly in~\cite{nekrashevych_p:scale_invariant}, where many other examples where produced. We call a group \emph{scale-invariant} if there is a sequence of finite index subgroups of $G$ that are all isomorphic to $G$ and whose intersection is trivial. Scale-invariant groups may be interesting for problems related to random walks, spectral theory of groups and graphs, statistical physics and fractal geometry, so the question of finding essentially new examples of scale invariant groups is relevant (see Question~\ref{que:scaleinvar} at the end of article).

\begin{proposition}
\label{prop:stab_finite_index}
A self-similar self-replicating group acting essentially freely on $\partial T(X)$ is scale invariant.
\end{proposition}

\begin{proof}
Let $G$ be as described in the statement. Then for each vertex $u\in X^*$ consider the stabilizer $\St_G(u)$ of $u$ in $G$.
First of all, the index of $\St_G(u)$ cannot exceed $|X|^{|u|}$ (where $|\cdot|$ denotes the cardinality of the argument) as vertex $u$ cannot be moved by $G$ outside its level, which has $|X|^{|u|}$ vertices. Since $G$ is self-replicating, the canonical homomorphism $\phi_u\colon\St_G(u)\to G$ defined by $\phi_u(g)=g|_u$ is surjective. On the other hand, the kernel of this homomorphism is trivial since otherwise we would obtain a nonidentity element in the trivializer of $u$ in $G$ contradicting to the essential freeness of the action of $G$ on $\partial T(X)$ by Proposition~\ref{prop:loc_triv}. Therefore, $\St_G(u)$ is isomorphic to $G$.

Since the action of $G$ on $\partial T(X)$ is essentially free, the set of points in $\partial T_2$ that have trivial stabilizers in $G$ has full measure. Let $\omega\in\partial T(X)$ be a point in this set, so that $\St_G(\omega)=\{1\}$. Denote by $w_n$ be the prefix of $\omega$ of length $n$. Then by the above argument the sequence $\St_G(w_n)$ is a nested sequence of finite index subgroups of $G$ that are all isomorphic to $G$ and whose intersection coincides with $\St_G(\omega)$, which is trivial.
\end{proof}

The previous corollary gives a potential way to construct essentially new examples of scale-invariant groups and is a partial motivation for this paper (see Question~\ref{que:scaleinvar} in Section~\ref{sec:concluding}).

\section{Proof of the main theorem.}
\label{sec:proof}
The proof of the main theorem (Theorem~\ref{thm:main}) is subdivided into 5 subsections. All except two automata in the class under consideration generate either groups that act not essentially freely on $\partial T_2$, or groups that have been studied before in the literature. In the first case the problem reduces to finding a nonidentity element in the rigid stabilizer of the group, while in the second case there is no need to reconstruct the structure of the group from scratch. So in both cases the analysis of the group is quite short. First we filter automata that generate groups acting not essentially freely using Mikhailova systems method and brute force methods in Subsection~\ref{ssec:reduction}. Then we treat manually remaining groups whose structure has already been described (in~\cite{bondarenko_gkmnss:full_clas32_short}) in Subsection~\ref{ssec:manual}. The remaining two automata ([2193] and [2372]) generate the groups that have not been studied before and little was known about them. We completely describe the structure of these groups and prove that they act essentially freely on $\partial T_2$ in Subsections~\ref{ssec:2193} and~\ref{ssec:2372} respectively.

Our systematic search for groups that act essentially freely on $\partial T_2$ heavily uses results of~\cite{bondarenko_gkmnss:full_clas32_short}, in conjunction with computations performed using \verb"AutomGrp" package~\cite{muntyan_s:automgrp} developed by Y.~Muntyan and the second author for \verb"GAP" system~\cite{GAP4}. We also note that because of a large number of groups studied sometimes we will use the same names for elements of different groups. In other words, all names of variables and constants are to be considered ``local'' and defined for each group individually.\vspace{3mm}

\subsection{Reduction using Mikhailova system and brute force methods.}
\label{ssec:reduction}

We start from the list of all 194 non minimally symmetric automata (recall that this notion was introduced in Definition~\ref{def:minim_sym}:

\begin{verbatim}
[ 1, 730, 731, 734, 739, 740, 741, 743, 744, 747, 748, 749, 750, 752,
  753, 756, 766, 767, 768, 770, 771, 774, 775, 776, 777, 779, 780,
  783, 802, 803, 804, 806, 807, 810, 820, 821, 824, 838, 839, 840,
  842, 843, 846, 847, 848, 849, 851, 852, 855, 856, 857, 858, 860,
  861, 864, 865, 866, 869, 870, 874, 875, 876, 878, 879, 882, 883,
  884, 885, 887, 888, 891, 919, 920, 923, 924, 928, 929, 930, 932,
  933, 936, 937, 938, 939, 941, 942, 945, 955, 956, 957, 959, 960,
  963, 964, 965, 966, 968, 969, 972, 1090, 1091, 1094, 2190, 2193,
  2196, 2199, 2202, 2203, 2204, 2205, 2206, 2207, 2209, 2210, 2212,
  2213, 2214, 2226, 2229, 2232, 2233, 2234, 2236, 2237, 2239, 2240,
  2241, 2260, 2261, 2262, 2264, 2265, 2271, 2274, 2277, 2280, 2283,
  2284, 2285, 2286, 2287, 2293, 2294, 2295, 2307, 2313, 2320, 2322,
  2352, 2355, 2358, 2361, 2364, 2365, 2366, 2367, 2368, 2369, 2371,
  2372, 2374, 2375, 2376, 2388, 2391, 2394, 2395, 2396, 2398, 2399,
  2401, 2402, 2403, 2422, 2423, 2424, 2426, 2427, 2838, 2841, 2844,
  2847, 2850, 2851, 2852, 2853, 2854, 2860, 2861, 2862, 2874, 2880,
  2887, 2889 ]
\end{verbatim}

Firstly, we compute Mikhailova systems for all automata in the above list and filter out those automata, for which Mikhailova system produces a nonidentity element in the rigid stabilizer. The nontriviality of the elements listed below was checked by the program, but can be checked by hands as well. This allows us to reduce by 93 the number of automata that have to be checked. For each such automaton we list this element and its decomposition at the first level. Here is the list:

\renewcommand{\arraystretch}{0.1}
\begin{multicols}{2}
\noindent 741: $c^{-1}a^{-1}ba = (1, a^{-1}c^{-1}bc)$\\
744: $b^{-1}c^{-1}ba^{-1}ca = (1, a^{-1}c^{-1}ac)$\\
749: $c^{-1}ac^{-1}ba^{-1}c = (1, a^{-1}c)$\\
753: $b^{-1}aba^{-1}bc^{-1} = (1, a^{-1}bcb^{-1})$\\
776: $a^{-1}ba^{-1}c = (1, b^{-1}c)$\\
777: $c^{-1}b^{-1}a^2 = (1, a^{-1}b^{-1}ac)$\\
779: $c^{-1}ab^{-1}cba^{-1} = (1, a^{-1}bc^{-1}acb^{-1})$\\
840: $b^{-1}a^{-1}ca = (1, b^{-1}c^{-1}bc)$\\
843: $c^{-1}a^{-1}ba = (1, a^{-1}c^{-1}ac)$\\
849: $c^{-1} = (1, a^{-1})$\\
852: $c^{-1} = (1, a^{-1})$\\
856: $c^{-2}ac^{-1}bca^{-1}c = (1, a^{-1}b^{-1}cb)$\\
857: $c^{-2}ac^{-1}ac^{-1}aba^{-1}c = (1, a^{-1}c)$\\
858: $c^{-1}b^{-1}aba^{-1}b = (1, a^{-1}b^{-1}aca^{-1}b)$\\
860: $c^{-1}aba^{-1} = (1, a^{-1}bcb^{-1})$\\
861: $c^{-1}b^{-1}aba^{-1}b = (1, a^{-1}c)$\\
864: $c^{-1}b^{-1}aba^{-1}b = (1, a^{-1}b^{-1}cb)$\\
866: $c^{-1}ac^{-1}a^{-1}cba^{-1}ca^{-1}c = (1, b^{-1}c)$\\
869: $c^{-1}ac^{-1}a^{-1}cb = (1, a^{-1}b^{-1}ac)$\\
874: $c^{-1}b = (1, a^{-1}c)$\\
875: $c^{-1}ac^{-1}ac^{-1}b = (1, a^{-1}c)$\\
876: $c^{-1}b = (1, a^{-1}c)$\\
878: $c^{-1}b = (1, a^{-1}c)$\\
879: $c^{-1}b = (1, a^{-1}c)$\\
882: $c^{-1}b = (1, a^{-1}c)$\\
883: $c^{-1}b^{-1}ac^{-1}bca^{-1}c = (1, a^{-1}c^{-1}ab^{-1}cb)$\\
885: $c^{-1}b^{-1}aba^{-1}c = (1, a^{-1}c^{-1}ac)$\\
887: $c^{-1}ab^{-1}a^{-1}cb = (1, a^{-1}bc^{-1}b^{-1}ac)$\\
888: $c^{-1}ab^{-1}a = (1, a^{-1}b)$\\
920: $b^{-1}ab^{-1}cba^{-1}ba^{-1}ba^{-1} = (1, b^{-1}c)$\\
923: $b^{-1}ab^{-1}c^{-1}ba^{-1}b^2 = (1, a^{-1}c^{-1}ab)$\\
929: $c^{-1}a^{-1}ca^{-1}c = (1, a^{-1}c)$\\
933: $c^{-1}a^2 = (1, a^{-1}c)$\\
937: $c^{-1}b = (1, a^{-1}b)$\\
938: $c^{-1}b = (1, a^{-1}b)$\\
939: $c^{-1}bc^{-1}ac^{-1}a = (1, a^{-1}b)$\\
941: $c^{-1}b = (1, a^{-1}b)$\\
942: $aba^{-1}ba^{-2}c^{-1}b = (1, a^{-1}b)$\\
945: $c^{-1}b = (1, a^{-1}b)$\\
955: $c^{-2}ac^{-1}bca^{-1}c = (1, a^{-1}c^{-1}bc)$\\
956: $c^{-1}b^{-1}aba^{-1}b = (1, a^{-1}c^{-1}aba^{-1}c)$\\
957: $c^{-2}aba^{-1}c^{-1}ac^{-1}ac = (1, a^{-1}b)$\\
959: $c^{-1}b^{-1}aba^{-1}b = (1, a^{-1}c^{-1}bc)$\\
960: $aba^{-1}ba^{-2}c^{-1}aba^{-1} = (1, a^{-1}b)$\\
963: $c^{-1}aba^{-1} = (1, a^{-1}cbc^{-1})$\\
965: $c^{-1}b = (1, a^{-1}c)$\\
969: $c^{-1}b = (1, a^{-1}c)$\\
2199: $b^{-1}a = (1, b^{-1}c)$\\
2202: $cb^{-1}c^{-1}b = (1, ab^{-1}a^{-1}b)$\\
2203: $c^{-2}ab = (1, a^{-2}cb)$\\
2204: $c^{-1}ab^{-1}c = (1, a^{-1}c)$\\
2207: $a^{-1}b = (1, b^{-1}c)$\\
2209: $c^{-1}aca^{-1} = (1, a^{-1}bab^{-1})$\\
2210: $c^{-1}b^{-1}cb = (1, a^{-1}c^{-1}ac)$\\
2213: $c^{-1}b^{-1}cb = (1, a^{-1}c^{-1}ac)$\\
2234: $c^{-1}b^{-1}ac^{-1}a^2 = (1, a^{-1}c^{-1}b^2)$\\
2236: $a^{-1}b = (1, a^{-1}c)$\\
2239: $ca^{-2}cba^{-1} = (1, c^{-1}abc^{-1})$\\
2261: $c^{-1}a^{-1}ca = (1, a^{-1}b^{-1}ab)$\\
2271: $b^{-1}a = (1, a^{-1}c)$\\
2274: $c^{-3}bc^2b^{-3}c^3 = (1, a^{-2}b^2)$\\
2280: $b^{-2}a^2b^{-1}a^{-1}b^2 = (1, b^{-1}c)$\\
2283: $c^{-2}bac^{-2}bcb^{-1}cb^{-2}c^2 = (1, a^{-1}c)$\\
2284: $c^{-1}bca^{-1} = (1, bc^{-1})$\\
2285: $c^{-1}ac^{-1}b = (1, a^{-1}c)$\\
2287: $c^{-1}b^{-1}c^2a^{-1}c = (1, a^{-1}c^{-1}b^2)$\\
2293: $b^{-1}c^2a^{-1} = (1, c^{-1}b^2c^{-1})$\\
2295: $c^{-1}ab^{-1}c = (1, a^{-1}c)$\\
2307: $c^{-1}bc^{-1}a = (1, a^{-1}c)$\\
2322: $ba^{-1} = (1, bc^{-1})$\\
2355: $b^{-1}a^{-1}cb^{-1}cb = (1, a^{-1}b^{-1}ca)$\\
2361: $b^{-1}a = (1, b^{-1}c)$\\
2364: $c^{-1}ac^{-1}b = (1, a^{-1}bc^{-1}b)$\\
2365: $ac^{-1}ac^{-1}b^{-1}c^2a^{-1} = (1, b^{-1}c)$\\
2366: $ba^{-1} = (1, ac^{-1})$\\
2367: $aca^{-1}cb^{-1}a^{-1} = (1, cab^{-1}c^{-1})$\\
2369: $a^{-1}b = (1, b^{-1}c)$\\
2371: $ac^{-2}b = (1, bc^{-1})$\\
2375: $c^{-1}b^{-1}ca = (1, a^{-1}b)$\\
2395: $b^{-1}ca^{-1}c^{-1}b^2 = (1, a^{-2}bc)$\\
2396: $c^{-1}bc^{-1}a = (1, a^{-1}bc^{-1}b)$\\
2398: $a^{-1}b = (1, a^{-1}c)$\\
2399: $a^{-1}b = (1, b^{-1}c)$\\
2401: $c^{-1}bca^{-1} = (1, a^{-1}b)$\\
2402: $c^{-2}ba = (1, a^{-2}b^2)$\\
2403: $ba^{-1} = (1, bc^{-1})$\\
2423: $c^{-1}bc^{-1}a = (1, a^{-1}b)$\\
2427: $ab^{-1} = (1, bc^{-1})$\\
2841: $b^{-1}a^{-1}ba^{-1} = (1, a^{-1}b^{-1}ab^{-1})$\\
2847: $b^{-1}a = (1, b^{-1})$\\
2850: $b^{-1}a^2b^{-1}ab = (1, a^{-1}b^2)$\\
2851: $a^{-3}b = (1, a^{-2}b)$\\
2852: $ab^{-1} = (1, a^{-1})$\\
\end{multicols}

For the remaining 101 automata we applied a brute force in an attempt to find nonidentity elements in the rigid stabilizer of the first level up to length 5 using the function \verb"FindGroupElement" of \verb"AutomGrp" package. This allows us to eliminate the following automata.

\begin{center}
\begin{tabular}{ll}
739: $bc=(ba, 1)$&                        2205: $(ab)^2=(1, (cb)^2)$\\
740: $bc^{-1}=(ba^{-1}, 1)$&              2206: $ab=(1, ac)$\\
743: $bc=(ba, 1)$&                        2212: $abc^{-2}=(1, c^2a^{-2})$\\
747: $bc=(ba, 1)$&                        2214: $ab=(ca, 1)$\\
748: $bc=(ca, )$&                         2229: $ab=(cb, 1)$\\
750: $bc^{-1}=(ca^{-1}, 1)$&              2233: $bcb^{-1}c=(1, bab^{-1}a)$\\
752: $bc=(ca, 1)$&                        2237: $ab^{-1}=(bc^{-1}, 1)$\\
756: $bc=(ca, 1)$&                        2241: $ab=(cb, 1)$\\
775: $bcbc=(1, baba)$&                    2262: $ab^{-1}=(1, ac^{-1})$\\
780: $(bc^{-1})^2=((ca^{-1})^2, 1)$&      2265: $ab^{-1}=(1, bc^{-1})$\\
783: $(bc)^2=(1, (ba)^2)$&                2286: $ab^{-1}ab^{-1}=(1, ca^{-1}ca^{-1})$\\
838: $abac=((ab)^2, 1)$&                  2352: $ab^{-1}=(ca^{-1}, 1)$\\
842: $abac=(1, (ba)^2)$&                  2368: $ab=(1, ac)$\\
847: $c=(1, a)$&                          2376: $ab=(ca, 1)$\\
848: $c=(1, a)$&                          2391: $ab=(cb, 1)$\\
851: $c=(1, a)$&                          2424: $ab^{-1}=(1, ac^{-1})$\\
855: $c=(1, a)$&                          2838: $ab^{-1}=(a^{-1}, 1)$\\
964: $bc=(1, ca)$&                        2853: $(ab)^2=(1, b^2)$\\
966: $bc^{-1}=(1, ca^{-1})$&              2854: $ab=(1, a)$\\
968: $bc=(1, ca)$&                        2860: $ab=(a^2, 1)$\\
972: $bc=(1, ca)$&                        2862: $ab=(a, 1)$\\
2190: $ab^{-1}=(ca^{-1}, 1)$&             2889: $ab=(b, 1)$
\end{tabular}
\end{center}

The above reduction leaves the following 57 candidates for automata that generate groups acting essentially freely:
\begin{verbatim}
[ 1, 730, 731, 734, 766, 767, 768, 770, 771, 774, 802, 803, 804, 806,
  807, 810, 820, 821, 824, 839, 846, 865, 870, 884, 891, 919, 924,
  928, 930, 932, 936, 1090, 1091, 1094, 2193, 2196, 2226, 2232, 2240,
  2260, 2264, 2277, 2294, 2313, 2320, 2358, 2372, 2374, 2388, 2394,
  2422, 2426, 2844, 2861, 2874, 2880, 2887 ]
\end{verbatim}
In the next three subsections, we investigate these cases separately.
\vspace{3mm}

\subsection{Investigation of easy cases.}
\label{ssec:manual}
The format of this subsection is as follows. Some automata listed in the end of previous subsection generate isomorphic groups, for which the proof of essential freeness/non-freeness is identical. We then unite such groups into one case. Other automata are treated separately. We start each case by listing the numbers of automata from the list at the the end of previous subsection treated in this case (these numbers are given in bold font). Within each case we mean by $G$ the group generated by an automaton under consideration.\vspace{3mm}

\noindent{\bf 1.} The automaton number {\bf1} generates the trivial group which by definition acts essentially freely on $\partial T_2$.\vspace{3mm}

\noindent{\bf 730,734,766,770,774,2232,2264,2844,2880.} All automata in this list generate the Klein group of order 4 isomorphic to $(\Z/2\Z)\times(\Z/2\Z)$. Straightforward check reveals that no nonidentity element of any of these groups belongs to $\Rist_G(1)$. Thus by Proposition~\ref{prop:rist} these groups act essentially freely on $\partial T_2$.\vspace{3mm}

\noindent{\bf 731,767,768,804,1091,2861,2887.} All automata in this family generate groups isomorphic to $\Z$. We will prove now that if an automaton generates $G\cong\Z$, then the action of $G$ on $\partial T_2$ is essentially free. Suppose not, then by Proposition~\ref{prop:rist} and spherical transitivity of a group (as $\Z$ is infinite self-similar group acting on the binary tree, by Lemma 3 in~\cite{bondarenko_gkmnss:full_clas32_short} its action must be spherically transitive) there must be a nonidentity element $g=(1,g|_1)$ in $\Rist_G(1)$. Since $G$ is nontrivial, by self-similarity there must be an element $h=(h|_0,h|_1)\sigma\in G$ that acts nontrivially on the first level. Conjugating $g$ by $h$ yields
\[g^h=(h_1^{-1},h_0^{-1})\sigma\cdot (1,g|_1)\cdot (h|_0,h|_1)\sigma=(g|_1^{h|_1},1).\]
Since both $g$ and $g^h$ are nonidentity elements of $G\cong\Z$, there must be $n,m\in \Z-\{0\}$ such that $g^n=(g^h)^m$, which implies
\[(1,g|_1^m)=((g|_1^{h|_1})^m,1).\]
This is a contradiction because $g|_1\in G$ has an infinite order as each nonidentity element of $G$. Thus $G$ acts essentially freely on $\partial T_2$.\vspace{3mm}

\noindent{\bf 771.} The wreath recursion for $G_{771}$ is $a=(c,1)\sigma, c=(a,a)$ and the group $G$ it generates is isomorphic to $\Z^2$ freely generated by $a$ and $c$. Each element of a stabilizer of the first level can be written as $a^{2n}c^m$ for some $n,m\in Z$. Since
\[a^{2n}c^m=(c^na^m,c^na^m),\]
the only time this element belongs to $\Rist_G(1)$ is when $n=m=0$, i.e. $a^{2n}c^m=1$. Thus, $\Rist_G(1)$ is trivial and $G$ acts essentially freely on $\partial T_2$ by Proposition~\ref{prop:rist}.\vspace{3mm}

\noindent{\bf 802,806,810,2196,2260.} All automata in this list generate an abelian group of order 8 isomorphic to $(\Z/2\Z)\times(\Z/2\Z)\times(\Z/2\Z)$. Straightforward check reveals that no nonidentity element of any of these groups belongs to $\Rist_G(1)$. Thus by Proposition~\ref{prop:rist} these groups act essentially freely on $\partial T_2$.\vspace{3mm}

\noindent{\bf 803.} The wreath recursion for $G_{803}$  is $a=(b,a)\sigma, b=(c,c), c=(a,a)$ and the group $G$ it generates is isomorphic to $\Z^2$ freely generated by $a$ and $b$ (where $c=a^{-2}b^{-1}$). Each element of a stabilizer of the first level can be written as $a^{2n}b^m$ for some $n,m\in Z$. Since
\[a^{2n}b^m=(a^nb^nc^m,a^nb^nc^m)\]
and the sections at the vertices of the first level are equal, the only time this element belongs to $\Rist_G(1)$ is when these sections are trivial, i.e. $m=n=0$ and, hence, $a^{2n}b^m=1$. Thus, $\Rist_G(1)$ is trivial and $G$ acts essentially freely on $\partial T_2$ by Proposition~\ref{prop:rist}.\vspace{3mm}

\noindent{\bf 807.} The wreath recursion for $G_{807}$  is $a=(c,b)\sigma, b=(c,c), c=(a,a)$ and the group $G$ it generates is isomorphic to $\Z^2$ freely generated by $a$ and $c$ (where $b=a^{-2}c^{-2}$). Each element of a stabilizer of the first level can be written as $a^{2n}c^m$ for some $n,m\in Z$. Since
\[a^{2n}c^m=(c^nb^na^m,c^nb^na^m)\]
and the sections at the vertices of the first level are equal, the only time this element belongs to $\Rist_G(1)$ is when these sections are trivial, i.e.  i.e. $m=n=0$ and, hence, $a^{2n}b^m=1$. Thus, $\Rist_G(1)$ is trivial and $G$ acts essentially freely on $\partial T_2$ by Proposition~\ref{prop:rist}.\vspace{3mm}

\noindent{\bf 820,824,865,919,928,932,936,2226,2358,2394,2422,2874.} All automata in this family generate the infinite dihedral group $D_\infty$. We will prove that if an automaton generates $G\cong D_\infty$, then the action of $G$ on $\partial T_2$ is essentially free by the same method we used for automata generating $\Z$. Suppose not, then by Proposition~\ref{prop:rist} and spherical transitivity of a group (again, as $D_\infty$ is infinite self-similar group acting on the binary tree its action must be spherically transitive) there must be a nonidentity element $g=(1,g|_1)$ in $\Rist_G(1)$. Since $G$ is nontrivial, by self-similarity there must be an element $h=(h|_0,h|_1)\sigma\in G$ that acts nontrivially on the first level. Conjugating $g$ by $h$ yields
\[g^h=(h_1^{-1},h_0^{-1})\sigma\cdot (1,g|_1)\cdot (h|_0,h|_1)\sigma=(g|_1^{h|_1},1).\]
Both $g$ and $g^h$ are different nonidentity elements of $G\cong D_\infty$ that commute. All centralizers of nonidentity elements in $D_\infty$ are cyclic either of order 2 or infinite.
Since we have already two nonidentity elements in the $C_G(g)$, this subgroup has to be isomorphic to $\Z$. Hence, there must be $n,m\in \Z-\{0\}$ such that $g^n=(g^h)^m$, which implies
\[(1,g|_1^m)=((g|_1^{h|_1})^m,1).\]
This is a contradiction because $g|_1\in G$ has an infinite order as $g$ has an infinite order. Thus $G$ acts essentially freely on $\partial T_2$.
\vspace{3mm}

\noindent{\bf 821.} The group $G_{821}$ generated by this automaton is isomorphic to the lamplighter group $\LL\cong (\Z/2\Z)\wr\Z$ (see~\cite{gns00:automata}) and has the following presentation:
\begin{equation}
\label{eqn:lamp}
G\cong\langle a,b\ |\ [(b^{-1}a),(b^{-1}a)^{b^i}]=(b^{-1}a)^2=1,\ i\geq1\rangle,
\end{equation}
that can be obtained from the standard presentation $\langle x,y\ \mid\ [x,x^{y^i}]=x^2=1, i\geq 1\rangle$ of $\LL$ by Tietze transformations.

The Mikhailova system for this group is
\[\begin{array}{lcl}
b^{-1}aba^{-1}b&=&(a, b)\\
b&=&(b, a).\\
\end{array}\]

Therefore, by Proposition~\ref{prop:ess_free_aut} it is enough to prove that the map $\phi\colon F_2\to F_2$ defined by
\[\begin{array}{l}
\phi(a)=b,\\
\phi(b)=a,\\
\end{array}\]
induces an automorphism of $\LL$.

To prove that the relators in presentation~\eqref{eqn:lamp} are mapped by $\phi$ to the identity element we first show by induction that
\[(b^{-1}a)^{b^i}=(b^{-1}a)^{a^i}\]
for all $i\ge0$. For $i=0$ there is nothing to prove. The induction step is proved as follows:
\[(b^{-1}a)^{b^{i+1}}=\left((b^{-1}a)^{b^i}\right)^b=\left((b^{-1}a)^{a^i}\right)^b=\left((b^{-1}a)^{a^{i+1}}\right)^{a^{-1}b}=(b^{-1}a)^{a^{i+1}}.\]

Therefore, for the relators in presentation~\eqref{eqn:lamp} we have:
\[\begin{array}{l}
\phi([(b^{-1}a),(b^{-1}a)^{b^i}])=[(a^{-1}b),(a^{-1}b)^{a^i}]=[(a^{-1}b),(a^{-1}b)^{b^i}]=1,\\
\phi((b^{-1}a)^2)=(a^{-1}b)^2=1.
\end{array}\]

Thus $\phi$ induces a surjective endomorphism of $G$. Since $\LL$ is residually finite, it has a Hopf property, so $\phi$ must be an isomorphism.
\vspace{3mm}

We will use the argument above using the Hopf property several times in this subsection. For easy reference we will state it in the form of lemma.

\begin{lemma}
\label{lem:hopf}
If $\phi\colon G\to G$ is a surjective endomorphism of a self-similar group $G$, then $phi$ is an isomorphism.
\end{lemma}

\noindent{\bf 839.} The wreath recursion of this automaton is $a=(b,a)\sigma, b=(a,b), c=(b,a)$. Since $c=aba^{-1}$ we get $G=\langle a,b\rangle\cong \LL$. The proof that the action of $G$ on $\partial T_2$ is essentially free is now identical to the one for the automaton 821.\vspace{3mm}

\noindent{\bf 846.} This is an automaton called Bellaterra automaton generating a free product of three groups of order 2: $(\Z/2\Z)*(\Z/2\Z)*(\Z/2\Z)$~\cite{nekrash:self-similar,bondarenko_gkmnss:full_clas32_short}. The automaton $\A_{846}$ is bireversible, so by Proposition~\ref{prop:bireversible} the group it generates acts essentially freely on $\partial T_2$.\vspace{3mm}

\noindent{\bf 870.} This automaton generates a group isomorphic to the Baumslag-Solitar group $BS(1,3)$ (see~\cite{bondarenko_gkmnss:full_clas32_short}). It is proved (using Proposition~\ref{prop:ess_free_aut}) in Example~5.5 in~\cite{grigorch:dynamics11eng} that this group acts essentially freely on $\partial T_2$.\vspace{3mm}

\noindent{\bf 884.} The wreath recursion for $G_{884}$ is $a=(b,a)\sigma, b=(c,c), c=(b,a)$. The Mikhailova system for this group is

\[\begin{array}{lcl}
u:=c^{-1}a^2b^{-1}c^{-1}aba^{-1}c&=&(a, b)\\
c&=&(b, a)\\
b&=&(c, c)
\end{array}\]

Since $[b,c^{-1}a]=1$ in $G$, but $[a,c^{-1}b]\neq 1$ in $G$ we get that the rigid stabilizer of the first level contains a nonidentity element $[c,b^{-1}u]=([b,c^{-1}a],[a,c^{-1}b])=(1,[a,c^{-1}b])$. Thus the action on the boundary of the tree is not essentially free.\vspace{3mm}

\noindent{\bf 891.} The wreath recursion for $G_{891}$  is $a=(c,c)\sigma, b=(c,c), c=(b,a)$. It is shown in~\cite{bondarenko_gkmnss:full_clas32_short} that the group generated by this automaton is isomorphic to $\LL\rtimes(Z/2\Z)=\bigl((\Z/2\Z)\wr Z\bigr)\rtimes (\Z/2\Z)$, where $\LL\cong L:=\langle \xi=ca, \zeta=bc\rangle$, and $\Z/2\Z=\langle c\rangle$ acts on $L$ by inversion of $\xi$ and $\zeta$. It follows, that $G$ has the following presentation with respect to the generating set $\{\xi,\zeta,c\}$:
\begin{equation}
\label{eqn:lampC2}
G\cong\langle \xi,\zeta,c\ |\ \xi^c=\xi^{-1}, \zeta^c=\zeta^{-1}, c^2=(\zeta\xi)^2=[\zeta\xi,(\zeta\xi)^{\zeta^i}]=1,i\geq 1\rangle.
\end{equation}
This presentation by Tietze transformations (using expression of $\xi$ and $\zeta$ in terms of $a$, $b$ and $c$) can be converted to the following presentation:
\begin{multline*}
G\cong\langle a,b,c\ |\  c^{-1}cac=a^{-1}c^{-1}, c^{-1}bcc=c^{-1}b^{-1},\\ c^2=(bc^2a)^2=[bc^2a,(bc^2a)^{(bc)^i}]=1, i\ge1\rangle,
\end{multline*}
that simplifies to
\begin{equation}
\label{eqn:lampC2_2}
G\cong\langle a,b,c\ |\ [ba,(ba)^{(bc)^i}]=(ba)^2=a^2=b^2=c^2=1, i\ge1\rangle
\end{equation}

The Mikhailova system for this automaton is
\[\begin{array}{lcl}
b^{-1}aca^{-1}b&=&(a, b),\\
c&=&(b, a),\\
b&=&(c, c).
\end{array}\]

Therefore, by Proposition~\ref{prop:ess_free_aut} it is enough to prove that the map $\phi\colon F_3\to F_3$ defined by
\[\begin{array}{l}
\phi(a)=b,\\
\phi(b)=a,\\
\phi(c)=c
\end{array}\]
induces an automorphism of $G$.

To prove that the relators in presentation~\eqref{eqn:lampC2_2} are mapped by $\phi$ to the identity element we first show by induction that
\[(ab)^{(ac)^i}=(ab)^{(bc)^i}\]
for all $i\ge0$. For $i=0$ the statement obviously holds. The induction step is proved as follows:
\[(ab)^{(ac)^{i+1}}=\left((ab)^{(ac)^i}\right)^{ac}=\left((ab)^{(bc)^i}\right)^{ac}=\left((ab)^{(bc)^{i+1}}\right)^{(ab)^{bc}}=(ab)^{(bc)^{i+1}}.\]

Therefore, for the relators in presentation~\eqref{eqn:lamp} we have:
\[\begin{array}{l}
\phi([ba,(ba)^{(bc)^i}])=[ab,(ab)^{(ac)^i}]=[(ab),(ab)^{(bc)^i}]=1,\\
\phi((ba)^2)=(ab)^2=1,\\
\phi(a^2)=b^2=1,\quad
\phi(b^2)=a^2=1,\quad
\phi(c^2)=c^2=1.
\end{array}\]

Thus $\phi$ induces a surjective endomorphism of $G$ and by Lemma~\ref{lem:hopf} $\phi$ must be an isomorphism.\vspace{3mm}

\noindent{\bf 924.} The group generated by this automaton is isomorphic to $BS(1,3)$~\cite{bartholdi_s:bsolitar} and has the following presentation:
\[G\cong\langle a,b,c\ |\ (ac^{-1})^a(ac^{-1})^{-3}=ba^{-1}ca^{-1}=1\rangle,\]
that can be obtained from the standard presentation of $BS(1,3)$ by Tietze transformations.

The Mikhailova system for this automaton is
\[\begin{array}{lcl}
c^{-1}aca^{-1}c&=&(a, a^{-1}bcb^{-1}a),\\
c^{-1}a^2&=&(b, a^{-1}bc),\\
c&=&(c, a),\\
c^{-1}ab^{-1}a&=&(1, a^{-1}ba^{-1}c),\\
c^{-1}ac^{-1}b^{-1}ab&=&(1, a^{-1}bc^{-1}a^{-1}cb),
\end{array}\]
where $a^{-1}ba^{-1}c=a^{-1}bc^{-1}a^{-1}cb=1$ in $G$.

Therefore, by Proposition~\ref{prop:ess_free_aut} it is enough to prove that the map $\phi\colon F_3\to F_3$ defined by
\[\begin{array}{l}
\phi(a)=a^{-1}bcb^{-1}a,\\
\phi(b)=a^{-1}bc,\\
\phi(c)=a
\end{array}\]
induces an automorphism of $G$. Now we use \verb"AutomGrp" package to verify that the relators of $G$ are mapped by $\phi$ to the identity element in $G$:

\begin{verbatim}
gap> A:=a^-1*b*c*b^-1*a;
a^-1*b*c*b^-1*a
gap> B:=a^-1*b*c;
a^-1*b*c
gap> C:=a;
a
gap> IsOne((A*C^-1)^A*(A*C^-1)^-3);
true
gap> IsOne(B*A^-1*C*A^-1);
true
\end{verbatim}

Thus $\phi$ induces a surjective endomorphism of $G$ and by Lemma~\ref{lem:hopf} $\phi$ must be an isomorphism.\vspace{3mm}

\noindent{\bf 930.} The wreath recursion for $G_{930}$ is $a=(c,a)\sigma, b=(b,b), c=(c,a)$. Since the state $b$ determines the identity state, the group generated by this automaton coincides with the lamplighter group generated by automaton 821, which acts essentially freely on $\partial T_2$.\vspace{3mm}

\noindent{\bf 1090, 1094.} Both these automata generate a group of order 2, whose nonidentity element does not belong to the rigid stabilizer as it has to act nontrivially on the vertices of the first level (otherwise by self-similarity the group would contain more than 2 elements). Thus by Proposition~\ref{prop:rist} these groups act essentially freely on $\partial T_2$.\vspace{3mm}

\noindent{\bf 2240.} This is an Aleshin automaton (originally constructed in~\cite{aleshin:free}) generating a free group $F_3$ of rank 3~\cite{vorobets:aleshin}. The automaton itself is bireversible, so by Proposition~\ref{prop:bireversible} the group it generates acts essentially freely on $\partial T_2$.\vspace{3mm}

\noindent{\bf 2277.} The wreath recursion for $G_{2277}$ is $a=(c,c)\sigma, b=(a,a)\sigma, c=(b,a)$ and the group $G$ it generates is isomorphic to $\Z^2\rtimes (\Z/2\Z)$ as shown in~\cite{bondarenko_gkmnss:full_clas32_short}. More precisely, elements $x=bc$ and $y=ba$ freely generate $\Z^2$ and
\[G\cong \langle x,y\rangle\rtimes \langle b\rangle,\]
where $b$ is an element of order 2 acting nontrivially on the first level and acting on $\langle x,y\rangle$ by conjugation inverting each element.

Consider first elements in $\langle x,y\rangle$. We have the following wreath recursion for $x$ and $y$:
\[\begin{array}{lcl}
x&=&(1,y^{-1})\sigma,\\
y&=&(xy^{-1},xy^{-1}).
\end{array}
\]

Each element of a stabilizer of the first level of $\langle x, y\rangle$ can be written as $x^{2n}y^m$ for some $n,m\in Z$. Since
\[x^{2n}y^m=(x^my^{-m-n},x^my^{-m-n})\]
and the sections at the vertices of the first level are equal, the only time this element belongs to $\Rist_G(1)$ is when these sections are trivial, i.e. $x^{2n}y^m=1$.

Each element in $\St_G(1)$ which is not in $\langle x,y\rangle$ can be written as
\[x^{2n+1}y^mb=(x^my^{-m-n}a,x^my^{-m-n-1}a).\]
Both of the sections of the latter element are nontrivial since $a\notin\langle x,y\rangle$. Therefore, this element cannot belong to $\Rist_G(1)$.

Thus, $\Rist_G(1)$ is trivial and $G$ acts essentially freely on $\partial T_2$ by Proposition~\ref{prop:rist}.\vspace{3mm}

\noindent{\bf 2294.} The group generated by this automaton is isomorphic to $BS(1,-3)$ with the following presentation with respect to generators $a,b,c$ (see~\cite{bondarenko_gkmnss:full_clas32_short}):
\[G\cong\langle a,b,c\ |\ a^{-1}(a^{-1}c)a(a^{-1}c)^3=ca^{-1}cb^{-1}=1\rangle,\]
that can be obtained from the standard presentation of $BS(1,-3)$ by Tietze transformations.

The Mikhailova system for this automaton is
\[\begin{array}{lcl}
c^{-1}aca^{-1}cba^{-1}ca^{-2}c&=&(a, a^{-1}cbc^{-1}a^2c^{-1}ab^{-1}c^{-1}a),\\
c&=&(b, a),\\
c^{-1}a^2ba^{-1}ca^{-2}c&=&(c, a^{-1}cbac^{-1}ab^{-1}c^{-1}a),\\
c^{-1}ac^{-1}b&=&(1, a^{-1}cb^{-1}c),\\
c^{-2}a^2c^{-1}ab^{-1}c&=&(1, a^{-2}cba^{-1}c),
\end{array}\]
where $a^{-1}cb^{-1}c=a^{-2}cba^{-1}c=1$ in $G$.

Now by Proposition~\ref{prop:ess_free_aut} it is enough to prove that the map $\phi\colon F_3\to F_3$ defined by
\[\begin{array}{l}
\phi(a)=a^{-1}cbc^{-1}a^2c^{-1}ab^{-1}c^{-1}a,\\
\phi(b)=a,\\
\phi(c)=a^{-1}cbac^{-1}ab^{-1}c^{-1}a
\end{array}\]
induces an automorphism of $G$. We use \verb"AutomGrp" package to verify that the relators of $G$ are mapped to the identity element in $G$:
\begin{verbatim}
gap> G:=AutomatonGroup("a=(b,c)(1,2),b=(c,a)(1,2),c=(b,a)");
< a, b, c >
gap> A:=a^-1*c*b*c^-1*a^2*c^-1*a*b^-1*c^-1*a;
a^-1*c*b*c^-1*a^2*c^-1*a*b^-1*c^-1*a
gap> B:=a;
a
gap> C:=a^-1*c*b*a*c^-1*a*b^-1*c^-1*a;
a^-1*c*b*a*c^-1*a*b^-1*c^-1*a
gap> IsOne(A^-2*C*A*(A^-1*C)^3);
true
gap> IsOne(C*A^-1*C*B^-1);
true
\end{verbatim}

Thus, $\phi$ induces a surjective endomorphism of $G$ and by Lemma~\ref{lem:hopf} $\phi$ must be an isomorphism.\vspace{3mm}

\noindent{\bf 2313.} The wreath recursion for $G_{2313}$ is $a=(c,c)\sigma, b=(b,b)\sigma, c=(b,a)$ and the group $G$ it generates is isomorphic to $\Z^2\rtimes (\Z/2\Z)$ as shown in~\cite{bondarenko_gkmnss:full_clas32_short}. Elements $x=ab$ and $y=cb$ freely generate $\Z^2$ and
\[G\cong \langle x,y\rangle\rtimes \langle b\rangle,\]
where $b$ is an element of order 2 acting nontrivially on the first level and acting on $\langle x,y\rangle$ by conjugation inverting each element.

Consider first elements in $\langle x,y\rangle$. We have the following wreath recursion for $x$ and $y$:
\[\begin{array}{lcl}
x&=&(y,y),\\
y&=&(1,x)\sigma.
\end{array}
\]
This recursion coincides with the definition of an automaton $771$, which generates a group acting essentially freely on $\partial T_2$.
On the other hand, $G$ can be defined by wreath recursion $x=(y,y),y=(1,x)\sigma, b=(b,b)\sigma$. Since $b$ has order $2$, each element $g$ in the complement of $\langle x,y\rangle$ in $G$ can be written as $wb$, where $w\in\langle x,y\rangle$. Both sections of $g$ will be words in $x,y$ and $b$ containing exactly one $b$ (since $b=(b,b)\sigma$). Thus, these sections cannot be trivial since $b\notin\langle x,y\rangle$.

Therefore, the group $G$ also acts essentially freely on $\partial T_2$.\vspace{3mm}

\noindent{\bf 2320.} The group generated by this automaton is also isomorphic to $BS(1,-3)$ with the following presentation with respect to generators $a,b,c$:
\[G\cong\langle a,b,c\ |\ a(c^{-1}a)a^{-1}(c^{-1}a)^3=ca^{-1}cb^{-1}=1\rangle,\]
that can be obtained from the standard presentation of $BS(1,-3)$ by Tietze transformations.

Indeed, these relations do hold in $G$:
\begin{verbatim}
gap> G:=AutomatonGroup("a=(a,c)(1,2),b=(c,b)(1,2),c=(b,a)");
< a, b, c >
gap> IsOne((c^-1*a)^(a^-1)*(c^-1*a)^3);
true
gap> IsOne(c*a^-1*c*b^-1);
true
\end{verbatim}

And since both $a$ and $c^{-1}a$ are of infinite order:
\begin{verbatim}
gap> Order(a);
infinity
gap> Order(c^-1*a);
infinity
\end{verbatim}
we have an isomorphism $G\cong BS(1,-3)$.

The Mikhailova system for this automaton is
\[\begin{array}{lcl}
aca^{-1}&=&(a, cbc^{-1}),\\
c&=&(b, a),\\
bc^{-1}bc^{-1}aca^{-1}&=&(c, ca^{-1}cbc^{-1}),\\
c^{-1}ac^{-1}b&=&(1, a^{-1}cb^{-1}c),\\
a^{-1}ca^{-1}bc^{-1}bc^{-1}aca^{-1}&=&(1, a^{-1}ba^{-1}cbc^{-1}),\\
\end{array}\]
where $a^{-1}cb^{-1}c=a^{-1}ba^{-1}cbc^{-1}=1$ in $G$.

Now by Proposition~\ref{prop:ess_free_aut} it is enough to prove that the map $\phi\colon F_3\to F_3$ defined by
\[\begin{array}{l}
\phi(a)=cbc^{-1},\\
\phi(b)=a,\\
\phi(c)=ca^{-1}cbc^{-1}
\end{array}\]
induces an automorphism of $G$. We use \verb"AutomGrp" package to verify that the relators of $G$ are mapped to the identity element in $G$:
\begin{verbatim}
gap> A:=c*b*c^-1;
c*b*c^-1
gap> B:=a;
a
gap> C:=c*a^-1*c*b*c^-1;
c*a^-1*c*b*c^-1
gap> IsOne((C^-1*A)^(A^-1)*(C^-1*A)^3);
true
gap> IsOne(C*A^-1*C*B^-1);
true
\end{verbatim}

Thus, $\phi$ induces an surjective endomorphism of $G$ and by Lemma~\ref{lem:hopf} $\phi$ must be an isomorphism.\vspace{3mm}

\noindent{\bf 2374.} The wreath recursion for $G_{2374}$ is $a=(a,c)\sigma, b=(c,a)\sigma, c=(c,a)$. Since $b=cac^{-1}$, the group generated by this automaton coincides as a subgroup of $\Aut(T_2)$ with the lamplighter group generated by automaton 930, which acts essentially freely on $\partial T_2$.
\vspace{3mm}

\noindent{\bf 2388.} The wreath recursion for $G_{2388}$ is $a=(c,a)\sigma, b=(b,b)\sigma, c=(c,a)$. Since $b=\sigma=c^{-1}a$, the group generated by this automaton coincides as a subgroup of $\Aut(T_2)$ with the lamplighter group generated by automaton 821, which act essentially freely on $\partial T_2$.
\vspace{3mm}

\noindent{\bf 2426.} The wreath recursion for $G_{2426}$ is $a=(b,b)\sigma, b=(c,c)\sigma, c=(c,a)$ and the group itself is isomorphic to $\Z^2\rtimes (\Z/2\Z)$ as shown in~\cite{bondarenko_gkmnss:full_clas32_short}. The proof of essential freeness is identical to the one for the automaton 2277. The elements $x=ba$ and $y=bc$ freely generate $\Z^2$ and
\[G\cong \langle x,y\rangle\rtimes \langle b\rangle,\]
where $b$ is an element of order 2 acting nontrivially on the first level.

Consider first elements in $\langle x,y\rangle$. We have the following wreath recursion for $x$ and $y$:
\[\begin{array}{lcl}
x&=&(y^{-1},y^{-1}),\\
y&=&(y^{-1}x,1)\sigma.
\end{array}
\]

Each element of the stabilizer of the first level of $\langle x, y\rangle$ can be written as $x^{n}y^{2m}$ for some $n,m\in Z$. Since
\[x^{n}y^{2m}=(x^my^{-m-n},x^my^{-m-n})\]
and the sections at the vertices of the first level are equal, the only time this element belongs to $\Rist_G(1)$ is when these sections are trivial, i.e. $x^{n}y^{2m}=1$.

On the other hand, both sections of each element in the complement of $\langle x,y\rangle$ in $G$ will be words in $x,y$ and $c$ containing exactly one $c$. Thus, these sections cannot be trivial since $c\notin\langle x,y\rangle$.

Thus, $\Rist_G(1)$ is trivial and $G$ acts essentially freely on $\partial T_2$ by Proposition~\ref{prop:rist}.

The only two remaining automata to consider are automata $\A_{2193}$ and $\A_{2372}$. We devote the next two subsections to the complete analysis of these two special cases.\vspace{3mm}

\subsection{Automaton 2193}
\label{ssec:2193}
Throughout this subsection we denote by $G$ the group $G_{2193}$ generated by automaton $\A_{2193}$ and defined by the following wreath recursion: $a=(c,b)\sigma$, $b=(a,a)\sigma$, $c=(a,a)$. The automaton $\A_{2193}$ itself is depicted in Figure~\ref{fig:2193}. Our goal in this subsection is to prove the following structure theorem for $G$, that will allow us to prove that the action of $G$ on $\partial T_2$ is essentially free in Corollary~\ref{cor:essfree2193}.

\begin{figure}
\begin{center}
\begin{picture}(1450,1090)(0,130)
\put(200,200){\circle{200}} 
\put(1200,200){\circle{200}}
\put(700,1070){\circle{200}}
\allinethickness{0.2mm} \put(45,280){$a$} \put(1280,280){$b$}
\put(820,1035){$c$}
\put(164,165){$\sigma$}  
\put(1164,165){$\sigma$}  
\put(634,1022){$id$}       
\put(300,200){\line(1,0){800}} 
\path(1050,225)(1100,200)(1050,175)   
\spline(200,300)(277,733)(613,1020)   
\path(559,1007)(613,1020)(591,969)    
\spline(287,150)(700,0)(1113,150)     
\path(325,109)(287,150)(343,156)      
\spline(650,983)(250,287)      
\path(297,318)(250,287)(253,343)      
\put(230,700){$_0$} 
\put(680,240){$_1$} 
\put(650,87){$_{0,1}$}   
\put(455,585){$_{0,1}$}  
\end{picture}
\caption{Automaton $\A_{2193}$ generating $G_{2193}$\label{fig:2193}}
\end{center}
\end{figure}
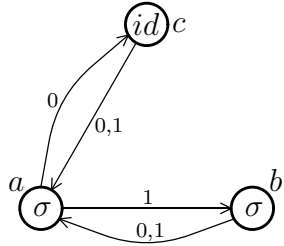

\begin{theorem}
\label{thm:structure2193}
The group $G=\langle a,b,c\rangle=\langle a^2, b^{-1}c, b^{-1}a, ac^{-1}a\rangle$ is solvable of derived length 3 and has the following structure:
\[G\cong \LL_{2,2}\rtimes (\Z/2\Z)\cong\bigl((\Z/2\Z)^2\wr \Z\bigr)\rtimes (\Z/2\Z),\]
where the isomorphism is induced by sending the first two generators $a^2$, $b^{-1}c$ from the second generating set of $G$ to generators of the base group $(\Z/2\Z)^2$ in $\LL_{2,2}$, the generator $b^{-1}a$ to the generator of $\Z$ in $\LL_{2,2}$, and the generator $t:=ac^{-1}a$ of $G$ to the generator of $\Z/2\Z$ in the semidirect product $\LL_{2,2}\rtimes (\Z/2\Z)$ acting on $\LL_{2,2}$ according to the following rules:
\begin{equation}
\label{eqn:conj_by_t_thm}
\begin{array}{l}
(b^{-1}c)^t=(b^{-1}a)^{-1}(b^{-1}c)a^2y^{-1}(b^{-1}c)^{-1}(b^{-1}a)a^2(b^{-1}c)^{-1}(b^{-1}a),\\
(b^{-1}a)^t=(b^{-1}a)^{-1}(b^{-1}c)a^2y^{-1}(b^{-1}c)^{-1}(b^{-1}a),\\
(a^2)^{t}=(b^{-1}a)^{-1}(b^{-1}c)a^2(b^{-1}c)^{-1}(b^{-1}a).
\end{array}
\end{equation}
Moreover, the group $G$ has the following presentation:
\begin{multline}
\label{eqn:presentation2193}
G\cong\langle a,b,c\ |\   a^4=(b^{-1}c)^2=1,\\ \left[a^2,(a^2)^{(b^{-1}a)^i}\right]=\left[a^2,(b^{-1}c)^{(b^{-1}a)^i}\right]=\left[b^{-1}c,(b^{-1}c)^{(b^{-1}a)^i}\right]=1,\ i\ge1,\\
(ba^2)^2=(ca^2)^2=1\rangle
\end{multline}
\end{theorem}

We begin from the introduction of necessary notation and technical lemmas. It is shown in~\cite{bondarenko_gkmnss:full_clas32_short} that a group $L=\langle x=a^{-1}c, y=b^{-1}a\rangle$ is isomorphic to the lamplighter group $\LL$, and this group acts on $X^*$ in a self-similar way via the following wreath recursion:
\begin{equation}
\begin{array}{lcll}
x&=&(y\phantom{^{-1}},&x^{-1})\sigma,\\
y&=&(y^{-1},&x\phantom{^{-1}}).
\end{array}
\end{equation}

Below, we will use the \verb"GAP" package \verb"AutomGrp"~\cite{muntyan_s:automgrp}. For the convenience of the reader, in nontrivial cases we will provide a code used to obtain the results. We start from encoding $G$, together with extra generators $x$ and $y$, in \verb"AutomGrp":

\begin{verbatim}
gap> L:=SelfSimilarGroup("a=(c,b)(1,2), b=(a,a)(1,2), c=(a,a),\
>                         x=(y,x^-1)(1,2), y=(y^-1,x)");
< a, b, c, x, y >
\end{verbatim}

We first observe that the following relations hold in $G$ (as can be verified either by hands or using \verb"IsOne" or \verb"FindGroupRelation" commands in \verb"AutomGrp"):
\begin{eqnarray}
\label{eqn:abel1}a^4=b^4=c^4=[b,c]=(cb^{-1})^2=1,\\
\label{eqn:abel2}b^ab^{a^{-1}}=c^ac^{a^{-1}}=a^ca^{b^{-1}}=1.
\end{eqnarray}

\begin{lemma}
The derived subgroup $G'$ of $G$ has index 8 in $G$ and the abelianization $G/G'$ of $G$ is isomorphic to $(\Z/2\Z)^3$.
\end{lemma}

\begin{proof}
It follows from~\eqref{eqn:abel2} that the images of generators $a,b,c$ in the abelianization $G/G'$ all have order $2$. Thus, $G/G'$ may have at most 8 elements and the commutator subgroup $G'$ has index at most $8$ in $G$. On the other hand, by looking at the third level of the tree we deduce that this index has to be at least 8. Indeed, if $\Stg(3)$ denotes a normal subgroup of $G$ consisting of all elements stabilizing all vertices of the third level of the tree and $\chi\colon G\to G/\Stg(3)$ is a canonical epimorphism, then $\chi(G')=(G/\Stg(3))'$ and $G'<\chi^{-1}((G/\Stg(3))')$. Now since $[G/\Stg(3):(G/\Stg(3))']=8$:
\begin{verbatim}
gap> Size(PermGroupOnLevel(G,3));
64
gap> Size(DerivedSubgroup(PermGroupOnLevel(G,3)));
8
\end{verbatim}
we get that the index of $G'$ in $G$ is at least the index of $\chi^{-1}((G/\Stg(3))')$ in $G$, which is equal to 8.
\end{proof}

The Reidemeister-Schreier procedure with the system of coset representatives $T=\{1,a,b,c,ab,ac,bc,abc\}$ yields:
\begin{equation}
\label{eqn:commutant}
G'=\langle a^2, [b^{-1},a^{-1}],[c^{-1},a^{-1}]\rangle=\langle a^2, [a,b],[a,c].\rangle
\end{equation}
Moreover, this generating set is minimal, what can be seen already on the third level of the tree while passing to corresponding finite quotients.

Consider the subgroup $H$ of $G$ defined by
\[H=\langle a^2, x,y\rangle.\]
As we will use $H$ in the computations below, we also encode it in \verb"AutomGrp".
\begin{verbatim}
gap> H := Group(a^2, x, y);
< a^2, x, y >
\end{verbatim}

\begin{proposition}
Subgroup $H$ is a subgroup of $G$ of index 2 (hence, $H$ is normal in $G$ and contains $G'$). Moreover, $G=\langle H,a\rangle$.
\end{proposition}

\begin{proof}
Since $[a,b]=a^2y^{-2}$ and $[a,c]=a^2x^2$, by Equation~\eqref{eqn:commutant} we get that $G'<H$. Further, since $G=\langle a, H\rangle$, in order to check that $H$ is normal in $G$ it is enough to check that $H$ is closed under the conjugation by $a$ and $a^{-1}$, which follows from the following identities:
\begin{equation}
\label{eqn:conj_by_a}
\begin{array}{l}
x^a=x^{-1}a^2,\\
y^a=a^2y^{-1},\\
x^{a^{-1}}=a^2x^{-1},\\
y^{a^{-1}}=y^{-1}a^2.
\end{array}
\end{equation}

Further, since the permutation induced by $a$ on the third level of the tree does not belong to the permutation group acting on the third level induced by $H$:
\begin{verbatim}
gap> PermOnLevel(a,3) in PermGroupOnLevel(H,3);
false
\end{verbatim}
we get that $a\notin H$. On the other hand, since $a^2\in H$ and $H$ is normal in $G$, we get that $H$ has index 2 in $G$.
\end{proof}

The following proposition completely describes the structure of $H$, and hence, $G$.

\begin{proposition}
\label{prop:L_2}
The group $H=\langle a^2, x, y\rangle=\langle a^2, yx, y\rangle$ is isomorphic to the rank 2 lamplighter group $\LL_{2,2}=(\Z/2\Z)^2\wr \Z$, where the isomorphism is induced by sending generators $a^2$ and $yx$ of $H$ to generators of $(\Z/2\Z)^2$ of $\LL_{2,2}$ and the generator $y\in H$ to the generator of $\Z$ in $\LL_{2,2}$. Moreover, $H$ is a self-similar group generated by the 6-state automaton depicted in Figure~\ref{fig:L22autom}.
\end{proposition}

\begin{figure}
\begin{center}
\includegraphics{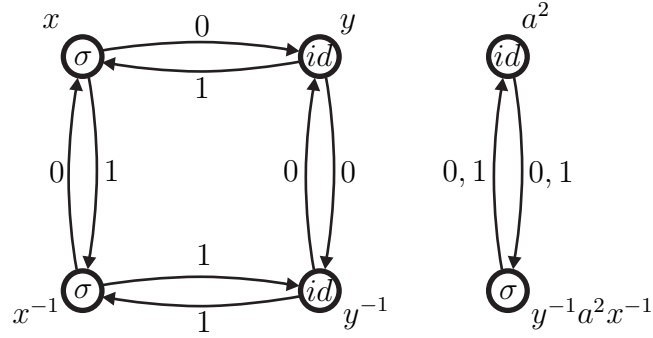}
\end{center}
\caption{Automaton generating the rank 2 lamplighter group $\LL_{2,2}$\label{fig:L22autom}}
\end{figure}

The strategy for the proof of this theorem is similar to the one used in~\cite{grigorch_z:lamplighter}, but is more general and involves more details. We start from an auxiliary definition.

\begin{definition}
An automorphism $g$ of the tree $X^*$ is called \emph{spherically homogeneous} if for each level $l$ the sections of $g$ at all vertices of $X^l$ act identically on the first level (or, equivalently, coincide).
\end{definition}

Every such automorphism can be defined by a sequence $\{\sigma_n\}_{n\geq 1}$ of permutations of $X$ where $\sigma_n$ describes the action of $g$ on the $n$-th letter of the input word over $X$. Given a sequence $(\sigma_n)_{n\geq 1}$ we will denote the corresponding spherically homogeneous automorphism by $[\sigma_n]_{n\geq 1}$ or simply as $[\sigma_1,\sigma_2,\sigma_3,\ldots]$.

Obviously, all spherically homogeneous automorphisms of $X^*$ form a group, which we denote by $\SHAut(X^*)$, isomorphic to a product of uncountably many copies of $\Sym(|X|)$. In the case of a binary tree, this group is abelian and isomorphic to the abelianization of $\Aut(T_2)$, which, in turn, is isomorphic to $\prod_{\mathbb N}\Z/2\Z$.

Below, we will prove that $H$ is contained in the normalizer of $\SHAut(X^*)$ in $\Aut(X^*)$, even though neither of generators $x$ or $y$ is spherically homogeneous. It is implicitly proved in~\cite{grigorch_z:lamplighter} that the standard representation of a lamplighter group in $\Aut(T_2)$ is contained in the normalizer of $\SHAut(X^*)$.

%
%

The following terminology is motivated by similar one in~\cite{grigorch_z:lamplighter}.

\begin{definition}[Generalized Conjugations]~Let $x$ and $y$ be as before.\\[-5mm]
\begin{itemize}
\item[(a)]
For an element $g\in\Aut(T_2)$, by a \emph{generalized elementary conjugation} of $g$ we mean the elements $y^{-1}gy$, $y^{-1}gx^{-1}$, $xgy$, $xgx^{-1}$, $ygy^{-1}$, $x^{-1}gy^{-1}$, $ygx$ and $x^{-1}gx$.
\item[(b)]
The first four elements in (a) are called \emph{positive elementary conjugations} and the latter four elements by \emph{negative elementary conjugations}.
\item[(c)] A composition of $k$ generalized (positive, negative) elementary conjugations is called a \emph{generalized (positive, negative) conjugation of length $k$}.
\end{itemize}
\end{definition}

For example, $y^{-1}y^{-1}x\cdot g\cdot x^{-1}yx^{-1}$ is a generalized positive conjugation of $g$ of length 3.

\begin{lemma}
\label{lem:seq_conj}
The generalized conjugations of spherically homogeneous automorphisms are spherically homogeneous.
\end{lemma}

\begin{proof}
By induction on the length of a generalized conjugation, it is enough to prove the lemma only for elementary generalized conjugations.

The key observation required for the proof is that both $xy=(xy)^{-1}$ and $yx=(yx)^{-1}$ are spherically homogeneous, and, consequently, commute with each $q\in \SHAut(X^*)$. Indeed, we have
\begin{equation}
\label{eqn:ABinv}
\begin{array}{l}
xy=[\sigma,\sigma,1,1,1,\ldots]\\
yx=[\sigma,1,1,1,1,\ldots]=\sigma\\
\end{array}
\end{equation}

Therefore, for each $q\in \SHAut(X^*)$ we have $xy\cdot q=q\cdot xy$ and hence,
\begin{equation}
\label{eqn:conj1}
yqy^{-1}=x^{-1}qx.
\end{equation}
Similarly, we get
\begin{equation}
\label{eqn:conj2}
\begin{array}{lcl}
y^{-1}qx^{-1}&=&xqy^{-1},\\
yqx&=&x^{-1}qy^{-1},\\
yqy^{-1}&=&x^{-1}qx.
\end{array}
\end{equation}

Therefore, it is enough to consider only 4 elementary generalized conjugations of each element (2 positive and 2 negative). The statement of the lemma will follow by induction on the level of the tree from Equations~\eqref{eqn:conj1} and~\eqref{eqn:conj2}. More precisely, we will prove that the statement $P(l)=$``for each generalized conjugation of every $q\in\SHAut(X^*)$ all its sections at all vertices of $X^l$ act identically on the first level'' is true for all $l\geq0$. The base case $P(0)$ follows trivially as there is only one section of each generalized conjugation at the root of the tree. In the induction step we assume that $P(l_0)$ is true. Let $q\in\SHAut(X^*)$ be arbitrary element. Then we have either $q=(q',q')\sigma$ or $q=(q',q')$ for some $q'\in\SHAut(X^*)$. Consider these cases separately.

\noindent \textit{Case I.} $q=(q',q')\sigma$. Then
\begin{equation}
\label{eqn:conj_identities1}
\begin{array}{lllll}
y^{-1}qy&=&(yq'x, x^{-1}q'y^{-1})\sigma&=&(yq'x,yq'x)\sigma,\\
y^{-1}qx^{-1}&=&(yq'y^{-1}, x^{-1}q'x)&=&(yq'y^{-1}, yq'y^{-1}),\\
yqx&=&(y^{-1}q'x^{-1},xq'y)&=&(xq'y,xq'y),\\
yqy^{-1}&=&(y^{-1}q'x^{-1},xq'y)\sigma&=&(xq'y,xq'y)\sigma.
\end{array}
\end{equation}

\noindent \textit{Case II.} $q=(q',q')$. Then
\begin{equation}
\label{eqn:conj_identities2}
\begin{array}{lllll}
y^{-1}qy&=&(yq'y^{-1}, x^{-1}q'x)&=&(yq'y^{-1},yq'y^{-1}),\\
y^{-1}qx^{-1}&=&(yq'x, x^{-1}q'y^{-1})\sigma&=&(yq'x,yq'x)\sigma,\\
yqx&=&(y^{-1}q'y,xq'x^{-1})\sigma&=&(y^{-1}q'y,y^{-1}q'y)\sigma,\\
yqy^{-1}&=&(y^{-1}q'y,xq'x^{-1})&=&(y^{-1}q'y,y^{-1}q'y).
\end{array}
\end{equation}

In each case we see that the sections of generalized conjugations of $q$ on the vertices of the first level coincide and are themselves generalized conjugations of an element $q'\in\SHAut(X^*)$. Therefore, by induction assumption, for each generalized conjugation of $q$ all its sections at all vertices of $X^{l_0+1}$ act identically on the first level.

We also note that the sections of positive elementary generalized conjugations are negative elementary generalized conjugations and vice versa.
\end{proof}

Recall that $xy$ is spherically homogeneous. It is crucial for the arguments below that $a^2$ is spherically homogeneous as well (note that $a$ is not spherically homogeneous). It is straightforward to check that
\begin{equation}
\label{eqn:a2spherhomo}
a^2=[1,\sigma,1,\sigma,1,\sigma,1,\ldots],
\end{equation}
where 1's and $\sigma$'s alternate with level. Therefore, by Lemma~\ref{lem:seq_conj} all conjugates of $a^2$ and $xy$ by powers of $y$ are spherically homogeneous, and thus, all of them are involutions and commute with each other. To finish the proof of Proposition~\ref{prop:L_2} it is now enough to show that all these conjugates are different.

It is proved in~\cite[p.131]{bondarenko_gkmnss:full_clas32_short} that all conjugates of $yx$ by powers of $y$ are different and finitary (i.e. have nontrivial sections only up to some finite level). This automatically implies that $(a^2)^{y^i}\neq (yx)^{y^j}$ for any $i,j$. Indeed, if $(a^2)^{y^i}=(yx)^{y^j}$, then $a^2=(yx)^{y^{j-i}}$ must be finitary, which is not the case.

Thus, it is left to show that $(a^2)^{y^i}\neq(a^2)^{y^j}$ for $i\neq j$. For this, of course it suffices to construct an infinite number of different conjugates of $a^2$ by powers of $y$.

The fact that all conjugates of $yx$ by powers of $y$ are different was proved in~\cite{bondarenko_gkmnss:full_clas32_short} by explicitly computing the depth of $(yx)^{y^i}$ for all $i$, where the depth of a finitaty automorphism $h$ is the smallest level of the tree such that all sections of $h$ at the vertices of this level are trivial. In our case, even though the conjugates of $a^2$ are not finitary any more, the conjugates of $(a^2)^{y^{-1}}$ by positive powers of $y^3$ are ``antifinitaty'' in the following sense.

\begin{definition}
An automorphism $g$ of $T_2$ is called \emph{antifinitary} if there exists a level $k$ such that the sections of $g$ at all vertices of this level coincide with the automorphism $s=(s,s)\sigma=[\sigma,\sigma,\sigma,\ldots]$ that changes all letters in any input word to the opposite ones.

The smallest $k$ with the above property is called the \emph{antidepth} of $g$.
\end{definition}

The goal of the following lemmas is to show that the conjugates of $a^2$ by powers of $y$ are all different.

\begin{lemma}
\label{lem:qsigma}
If $g\in \SHAut(X^*)$ is a spherically homogeneous automorphism of $T_2$, then for each $v\in X^*$ the section of a generalized elementary conjugation of $g$ at $v$ is a generalized elementary conjugation of $g|_v$. Moreover, the positive and the negative conjugations alternate with the level.
\end{lemma}

\begin{proof}
For $|v|=1$ the statement follows from Equations~\eqref{eqn:conj_identities1} and~\eqref{eqn:conj_identities2}. Then the Lemma follows trivially by induction on $|v|$.
\end{proof}

As a direct corollary of the above lemma we obtain:

\begin{corollary}
\label{cor:qsigma}
If $g$ is a spherically homogeneous automorphism of $T_2$, then for each $v\in X^*$ of even length, the section of a generalized positive conjugation of $g$ of length $k$ at $v$ is a generalized positive conjugation of $g|_v$ of length $k$.
\end{corollary}

Define the following antifinitary automorphisms of $T_2$:
\[
\begin{array}{lll}
q&=&[\sigma,\sigma,1,\sigma,1,1,\sigma,\sigma,\sigma,\ldots],\\
w&=&[1,1,1,\sigma,1,1,\sigma,\sigma,\sigma,\ldots].\\
\end{array}\]

Since $g$ is spherically homogeneous, by Lemma~\ref{lem:seq_conj} all generalized elementary conjugations of $g$ are also spherically homogeneous. The next lemma exhibits more structure.

\begin{lemma}
\label{lem:sections_of_h}~\\[-5mm]
\begin{itemize}
\item[(a)]
For each positive generalized conjugation $h$ of $q$ of length 3, and for each $v\in X^6$
\[h|_v=w.\]

\item[(b)]
For each positive generalized conjugation $h$ of $w$ of length 3, and for each $v\in X^6$
\[h|_v=q.\]
\end{itemize}

\end{lemma}

\begin{proof}
We use \verb"AutomGrp" to check these identities. First, we define elements $q$ and $w$ in \verb"GAP". Since we are about to compute generalized conjugations of these elements, we will redefine the whole group $G$ by adding $q$,$w$, and their sections to the list of generators.
We note that as will be shown in Lemma~\ref{lem:sections_of_conjugates}, both $q$ and $w$ are elements of $G$, so since $G$ is self-similar, we do not change the whole group by doing this. We will not use this fact in future.

\begin{verbatim}
gap> G:=SelfSimilarGroup("a=(c,b)(1,2),b=(a,a)(1,2),c=(a,a),\
>                         x=(y,x^-1)(1,2),y=(y^-1,x),\
> w=(w1,w1),w1=(w2,w2),w2=(w3,w3),w3=(w4,w4)(1,2),\
>   w4=(w5,w5),w5=(w6,w6),w6=(w6,w6)(1,2),\
> q=(q1,q1)(1,2),q1=(q2,q2)(1,2),q2=(q3,q3),q3=(q4,q4)(1,2),\
>   q4=(q5,q5),q5=(q6,q6),q6=(q6,q6)(1,2)");
< a, b, c, x, y, w, w1, w2, w3, w4, w5, w6, q, q1, q2, q3,\
 q4, q5, q6 >
\end{verbatim}
There are only 2 different generalized positive elementary conjugations of each element (recall Equations~\eqref{eqn:conj1} and~\eqref{eqn:conj2}). Therefore, there are 8 potentially different generalized positive elementary conjugations of length 3. Below, we verify the statement of the lemma by checking all eight possible cases.

For (a) we have:

\begin{verbatim}
gap> Section(y^-3*q*y^3,[1,1,1,1,1,1])=w;
true
gap> Section(y^-3*q*y^2*x^-1,[1,1,1,1,1,1])=w;
true
gap> Section(y^-3*q*y*x^-1*y,[1,1,1,1,1,1])=w;
true
gap> Section(y^-3*q*y*x^-2,[1,1,1,1,1,1])=w;
true
gap> Section(y^-3*q*x^-1*y^2,[1,1,1,1,1,1])=w;
true
gap> Section(y^-3*q*x^-1*y*x^-1,[1,1,1,1,1,1])=w;
true
gap> Section(y^-3*q*x^-2*y,[1,1,1,1,1,1])=w;
true
gap> Section(y^-3*q*x^-3,[1,1,1,1,1,1])=w;
true
\end{verbatim}

Similarly for (b):

\begin{verbatim}
gap> Section(y^-3*w*y^3,[1,1,1,1,1,1])=q;
true
gap> Section(y^-3*w*y^2*x^-1,[1,1,1,1,1,1])=q;
true
gap> Section(y^-3*w*y*x^-1*y,[1,1,1,1,1,1])=q;
true
gap> Section(y^-3*w*y*x^-2,[1,1,1,1,1,1])=q;
true
gap> Section(y^-3*w*x^-1*y^2,[1,1,1,1,1,1])=q;
true
gap> Section(y^-3*w*x^-1*y*x^-1,[1,1,1,1,1,1])=q;
true
gap> Section(y^-3*w*x^-2*y,[1,1,1,1,1,1])=q;
true
gap> Section(y^-3*w*x^-3,[1,1,1,1,1,1])=q;
true
\end{verbatim}
This finishes the proof.
\end{proof}

\begin{lemma}
\label{lem:sections_of_conjugates}
For each $i\geq 1$ and $v\in X^{12i-8}$, we have $(a^2)^{y^{6i-1}}|_v=q$.
\end{lemma}

\begin{proof}
We proceed by induction on $i$. For $i=1$ we have:

\begin{verbatim}
gap> Section((a^2)^(y^5),[1,1,1,1])=q;
true
\end{verbatim}

The induction step follows from Lemmas~\ref{lem:qsigma} and~\ref{lem:sections_of_h}. Indeed, suppose $(a^2)^{y^{6i-1}}|_v=q$ for some $i$ and vertex $v=1^{12i-8}\in X^{12i-8}$ (recall that by Lemma~\ref{lem:seq_conj} all conjugates of $a^2$ by powers of $y$ are spherically homogeneous, so the section does not depend on the choice of $v$ in $X^{12i-8}$). Then by Corollary~\ref{cor:qsigma} $(a^2)^{y^{6i-1+3}}|_v$ is a positive (since $12i-8$ is even) generalized conjugation $h$ of length $3$ of $q$. Thus, by Lemma~\ref{lem:sections_of_h} (a)
\[(a^2)^{y^{6i-1+3}}|_{v1^6}=\bigl((a^2)^{y^{6i-1}})^{y^3}|_v\bigr)|_{1^6}=h|_{1^6}=w.\]
Repeating the same argument one more time and applying Lemma~\ref{lem:sections_of_h}(b) yields
\[(a^2)^{y^{6i-1+6}}|_{v1^{12}}=(a^2)^{y^{6(i+1)-1}}|_{1^{12(i+1)-8}}=q,\]
which finishes the proof.
\end{proof}

\begin{corollary}
\label{cor:antidepth}
For each $i\geq 1$ the antidepth of $(a^2)^{y^{6i-1}}$ is equal to $12i-2$. In particular, all conjugates of $a^2$ by powers of $y$ are different.
\end{corollary}

\begin{proof}
The first part immediately follows from Lemma~\ref{lem:sections_of_conjugates} and the fact that $(a^2)^{y^i}$ is spherically homogeneous by Lemma~\ref{lem:seq_conj}. Furthermore, if $(a^2)^{y^i}=(a^2)^{y^j}$ for some $i\neq j$, then there could be at most $|i-j|$ different conjugates of $a^2$ by powers of $y$, which contradicts to the first part.
\end{proof}

Now we have all the ingredients to prove Proposition~\ref{prop:L_2}.

\begin{proof}[Proof of Proposition~\ref{prop:L_2}]
We have already shown above that $(a^2)^{y^i}$ and $(yx)^{y^i}$, $i\in\Z$ all commute and have order 2. As was already mentioned, it was proved in~\cite{bondarenko_gkmnss:full_clas32_short} (automaton 891) that $L=\langle x,y\rangle$ is isomorphic to the lamplighter group and that $(yx)^{y^i}$ are all different and finitary. Corollary~\ref{cor:antidepth} guarantees that $(a^2)^{y^i}$ are distinct for all $i\in\Z$. So it remains to show that $(a^2)^{y^i}$ is not in $L$ for each $i\in\Z$. Since for each $i$ the order of $(a^2)^{y^i}$ is 2 (because it is a spherically homogeneous automorphism), this element could potentially be equal only to an element of the base group in $L$ isomorphic to $\oplus_{\Z}\Z/2\Z$ (because these are the only elements in the lamplighter group of order $2$), i.e., an element of the form $(yx)^{y^j}$. But, as indicated above, this is not possible since in this case $a^2$ would be finitary, which is not the case.

Thus, the group $\langle (a^2)^{y^i},(yx)^{y^j},\ i,j\in\Z\rangle$ is isomorphic to the infinite direct product of countably many copies of $(\Z/2\Z)^2$. The infinite cyclic group $\langle y\rangle$ acts on this product by conjugation, that corresponds to simply shifting the exponent of $y$. Consequently, the group $H=\langle a^2,x,y\rangle$ has a structure of the rank 2 lamplighter group
\[H\cong\LL_{2,2}\cong(\Z/2\Z)^2\wr \Z.\]
\end{proof}

Now we can proceed to the proof of the main theorem of this subsection.

\begin{proof}[Proof of Theorem~\ref{thm:structure2193}]
First of all, note that since metabelian group $H$ is a normal subgroup of index 2 in $G$, the group $G$ itself has a derived length at most 3. On the other hand, since $[[a,b],[a,c]]\neq 1$:
\begin{verbatim}
gap> IsOne(Comm(Comm(a,b),Comm(a,c)));
false
\end{verbatim}
the group $G$ cannot be metabelian and hence has derived length 3.

Recall that $G=\langle H, a\rangle$, the element $a$ has order 4 and $a^2\in H$. Therefore $G$ is not a semidirect product of $H$ and $\langle a\rangle$. However, the element $t=ax^{-1}=ac^{-1}a$ has order 2 and is certainly not in $H$ as $a\notin H$ and $x\in H$. Therefore,
\[G=H\rtimes\langle t\rangle\cong \bigl((\Z/2\Z)^2\wr \Z\bigr)\rtimes (\Z/2\Z),\]
where the action of $t$ on generators of $H$ is defined as
\begin{equation}
\label{eqn:conj_by_t}
\begin{array}{l}
x^t=(x^a)^{x^{-1}}=(x^{-1}a^2)^{x^{-1}}=a^2x^{-1},\\
y^t=(y^a)^{x^{-1}}=(a^2y^{-1})^{x^{-1}}=xa^2y^{-1}x^{-1},\\
(a^2)^{t}=\bigl((a^2)^{a}\bigr)^{x^{-1}}= (a^2)^{x^{-1}}=xa^2x^{-1},\\
\end{array}
\end{equation}
as follows from Equation~\eqref{eqn:conj_by_a}. Taking into account that $b^{-1}c=yx$ and $b^{-1}a=y$ produces equalities~\eqref{eqn:conj_by_t_thm}.

To get a presentation for $G$, we start from a presentation of $G$ coming from its structural description described above. Let $\xi=a^2$, $\eta=yx=b^{-1}c$, $y=b^{-1}a$ and $t=ac^{-1}a$ be the generators of $G$. Then $\LL_{2,2}=\langle\xi,\eta,y\rangle\lhd G$ has the following presentation as a rank 2 lamplighter group:
\[\LL_{2,2}\cong\langle\xi,\eta,y\ |\ \xi^2=\eta^2=1, [\xi,\xi^{y^i}]=[\xi,\eta^{y^i}]=[\eta,\eta^{y^i}]=1,\ i\ge1\rangle.\]
The action of $t$ on generators of $\LL_{2,2}$ follows from equations~\eqref{eqn:conj_by_t} and the identity $x=y^{-1}\eta$.

\begin{align}
&\xi^t=(a^2)^t=xa^2x^{-1}=y^{-1}\eta\xi\eta^{-1}y,\label{eqn:xi_t}\\
&\eta^t=(yx)^t=xa^2y^{-1}x^{-1}\cdot a^2x^{-1}=y^{-1}\eta\xi y^{-1}\eta^{-1}y\cdot\xi\eta^{-1}y,\label{eqn:eta_t}\\
&y^t=xa^2y^{-1}x^{-1}=y^{-1}\eta\xi y^{-1}\eta^{-1}y\label{eqn:y_t}.
\end{align}
Therefore the presentation for $G$ with respect to generators $\xi$, $\eta$, $y$ and $t$ is
\begin{multline}
\label{eqn:presentation2193_1}
G=\langle\xi,\eta,y,t\ |\ \xi^2=\eta^2=1, [\xi,\xi^{y^i}]=[\xi,\eta^{y^i}]=[\eta,\eta^{y^i}]=1,\ i\ge1,\\
t^2=1,\qquad \xi^t=y^{-1}\eta\xi\eta^{-1}y,\\
\eta^t=y^{-1}\eta\xi y^{-1}\eta^{-1}y\xi\eta^{-1}y, \qquad y^t=y^{-1}\eta\xi y^{-1}\eta^{-1}y\rangle.
\end{multline}
To finish the proof we only need to rewrite presentation~\eqref{eqn:presentation2193_1} in terms of generators $a$, $b$ and $c$. The relation in the first line of~\eqref{eqn:presentation2193_1} are rewritten simply by substituting $\xi=a^2$, $\eta=b^{-1}c$, $y=b^{-1}a$. These relations correspond precisely to the relations in the first two lines in the presentation~\eqref{eqn:presentation2193}.

The relation $t^2=(ac^{-1}a)^2=1$ is equivalent to
\begin{equation}
\label{eqn:rel1}
(ca^2)^2=1
\end{equation}
taking into account that $a^4=1$.

Further, relation~\eqref{eqn:xi_t} yields
\[(a^2)^{ac^{-1}a}=a^{-1}b\cdot b^{-1}c\cdot a^2\cdot c^{-1}a=a^{-1}ca^2c^{-1}a,\]
that trivially holds in a free group.

Relation~\eqref{eqn:y_t} is equivalent to
\[(b^{-1}a)^{ac^{-1}a}=a^{-1}b\cdot b^{-1}c\cdot a^2\cdot a^{-1}b\cdot c^{-1}b\cdot b^{-1}a=a^{-1}ca^{-1}\cdot a^2ba^{-1}\cdot ac^{-1}a,\]
that simplifies to $b^{-1}a=a^2ba^{-1}$ or, equivalently, to
\begin{equation}
\label{eqn:rel2}
(ba^2)^2=1.
\end{equation}

Finally, relation~\eqref{eqn:eta_t} is equivalent to
\begin{multline*}
(b^{-1}c)^{ac^{-1}a}=a^{-1}b\cdot b^{-1}c\cdot a^2\cdot a^{-1}b\cdot c^{-1}b\cdot b^{-1}a\cdot a^2\cdot c^{-1}b\cdot b^{-1}a=\\a^{-1}ca^{-1}\cdot a^2bc^{-1}a^2\cdot ac^{-1}a,
\end{multline*}
which again simplifies to $bc^{-1}=a^2bc^{-1}a^2$ and now follows trivially from relations~\eqref{eqn:rel1} and~\eqref{eqn:rel2}. This finishes the proof of the theorem.
\end{proof}

\begin{proposition}
\label{prop:auto2193}
The automorphism $\zeta$ of a free group $F(a,b,c)$ defined by $\zeta(a)=a$, $\zeta(b)=c$, $\zeta(c)=b$ induces an automorphism of $G$.
\end{proposition}

\begin{proof}
It is obvious that the images of relators in the first and the third lines of presentation~\eqref{eqn:presentation2193} of $G$ under $\zeta$ are again relators in $G$. To see that $\zeta$ sends relators in the second line of~\eqref{eqn:presentation2193} to the identity element of $G$ it is enough to notice that $c^{-1}a=(c^{-1}b)\cdot (b^{-1}a)$ and that $c^{-1}b=(b^{-1}c)^{-1}$ commutes with conjugates of $a^2$ and $(b^{-1}c)$ by powers of $b^{-1}a$. Indeed, we first prove by induction that
\[(a^2)^{(c^{-1}a)^i}=(a^2)^{(b^{-1}a)^i}\]
for all $i\ge0$. For $i=0$ there is nothing to prove; the induction step is proved as follows:
\[(a^2)^{(c^{-1}a)^{i+1}}=\bigl((a^2)^{(c^{-1}a)^{i}}\bigr)^{c^{-1}a}=\left(\bigl((a^2)^{(b^{-1}a)^{i}}\bigr)^{c^{-1}b}\right)^ {b^{-1}a}=(a^2)^{(b^{-1}a)^{i+1}}.\]
The same argument is also used to show that for all $i\ge1$
\[(b^{-1}c)^{(c^{-1}a)^i}=(b^{-1}c)^{(b^{-1}a)^i}.\]

Therefore, for the relators in the second line of~\eqref{eqn:presentation2193} we have:
\[
\begin{array}{l}
\zeta\left(\left[a^2,(a^2)^{(b^{-1}a)^i}\right]\right)=\left[a^2,(a^2)^{(c^{-1}a)^i}\right]=\left[a^2,(a^2)^{(b^{-1}a)^i}\right]=1\\
\zeta\left(\left[a^2,(b^{-1}c)^{(b^{-1}a)^i}\right]\right)=\left[a^2,(b^{-1}c)^{(c^{-1}a)^i}\right]=\left[a^2,(b^{-1}c)^{(b^{-1}a)^i}\right]=1\\
\zeta\left(\left[b^{-1}c,(b^{-1}c)^{(b^{-1}a)^i}\right]\right)=\left[b^{-1}c,(b^{-1}c)^{(c^{-1}a)^i}\right]=\left[b^{-1}c,(b^{-1}c)^{(b^{-1}a)^i}\right]=1
\end{array}
\]

Therefore, $\zeta$ induces an endomorphism of $G$. This endomorphism is obviously onto and also one-to-one since $\zeta$ is an involution.
\end{proof}

\begin{corollary}
\label{cor:essfree2193}
The group $G$ acts essentially freely on the boundary of the tree.
\end{corollary}

\begin{proof}
The stabilizer of the first level in $G$ is generated by
\[\begin{array}{lcl}
b^{-2}cbcb^{-1}c&=&(a,a),\\
cb^{-1}a&=&(b,c),\\
ac^{-1}b^{-1}c^2&=&(c,b).
\end{array}\]

In this situation Proposition~\ref{prop:auto2193} guarantees that we can apply Proposition~\ref{prop:ess_free_aut} and deduce that the action of $G$ on the boundary of the tree is essentially free.
\end{proof}

We end up this section with the following interesting observations.

\begin{proposition}
\label{prop:finitary_seq}
The group $A=\langle (yx)^{y^i},\ i\in\Z\rangle$ coincides with a group of all finitary spherically homogeneous automorphisms.
\end{proposition}

\begin{proof}
It is proved in~\cite{bondarenko_gkmnss:full_clas32_short} (see automaton 891) that all elements of the form $s_n=(yx)^{y^{-n}}$ are finitary spherically homogeneous automorphisms with depth $2n+1$ for nonnegative $n$ and $2(-n)$ for negative $n$. The propositions now immediately follows by induction on the level.
\end{proof}

\begin{proposition}~\\[-5mm]
\begin{itemize}
\item[(a)]
The subgroup $L=\langle x,y\rangle$ of $G$ has infinite index in $G$.
\item[(b)]
The closure $\bar L$ of $L$ has index 2 in the closure $\bar G$ of $G$.
\end{itemize}
\end{proposition}

\begin{proof}
(a) According to Theorem 3.5 in~\cite{grigorch_k:lamplighter} each subgroup of $\LL_{2,2}\cong H$ of finite index must be isomorphic to $\LL_{2,s}$ for some $s\geq 1$. Since $L=\langle x,y\rangle$ is isomorphic to a standard lamplighter group $\LL$, it cannot have a finite index in $H$, and thus in $G$.

(b) By proposition~\ref{prop:finitary_seq} the group of all spherically homogeneous automorphisms coincides with the closure $\bar A$ of $A$, where $A$ is from Proposition~\ref{prop:finitary_seq}. Thus, as by Equality~\eqref{eqn:a2spherhomo} $a^2$ is spherically homogeneous, $a^2\in\bar A<\bar L$. Therefore, $H=\langle a^2,x,y\rangle<\bar L$ and $\bar H<\bar L$. On the other hand, $L<H$ and so $\bar L<\bar H$ and $\bar L=\bar H$. Since $H$ has index 2 in $G$, $\bar L=\bar H$ has index at most 2 in $\bar G$. Finally, since $a$ induces a permutation of the third level of the tree that does not belong to the permutation group on this level induced by $H$, we must have that $a\notin\bar H$. Thus, $\bar L=\bar H$ has index 2 in $\bar G$.
\end{proof}
\vspace{3mm}

\subsection{Automaton 2372}
\label{ssec:2372}
Throughout this subsection let $G$ denote the group $G_{2372}$ generated by automaton $\A_{2372}$ and defined by the following wreath recursion: $a=(b,b)\sigma, b=(c,a)\sigma, c=(c,a)$. The automaton itself is shown in Figure~\ref{fig:2372}. We start from stating the main theorem of this subsection that will be the ground for the proof of essential freeness of the action of $G$ on $\partial T_2$.

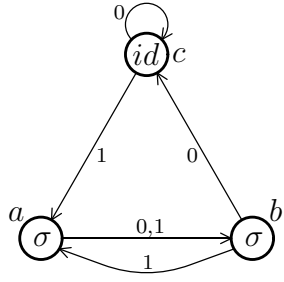
\begin{figure}
\begin{center}
\begin{picture}(1450,1090)(0,130)
\put(200,200){\circle{200}} 
\put(1200,200){\circle{200}}
\put(700,1070){\circle{200}}
\allinethickness{0.2mm} \put(45,280){$a$} \put(1280,280){$b$}
\put(820,1035){$c$}
\put(164,165){$\sigma$}  
\put(1164,165){$\sigma$}  
\put(634,1022){$id$}       
\put(300,200){\line(1,0){800}} 
\path(1050,225)(1100,200)(1050,175)   
\spline(287,150)(700,0)(1113,150)     
\path(325,109)(287,150)(343,156)      
\spline(750,983)(1150,287)     
\path(753,927)(750,983)(797,952)      
\spline(650,983)(250,287)      
\path(297,318)(250,287)(253,343)      
\put(700,1211){\arc{200}{2.36}{0.79}} 
\path(820,1168)(771,1141)(779,1196)   
\put(650,250){$_{0,1}$} 
\put(890,585){$_0$} 
\put(680,77){$_1$}   
\put(545,1261){$_0$}  
\put(460,585){$_1$}  
\end{picture}
\caption{Automaton $\A_{2372}$ generating $G_{2372}$\label{fig:2372}}
\end{center}
\end{figure}

\begin{theorem}
\label{thm:2372}
The group $G$, generated by states $a$, $b$ and $c$ of automaton $\A_{2372}$, is solvable of derived length 3 and has the following structure:
\begin{equation}
\label{eqn:presentation_txv}
G\cong BS(1,3)\rtimes (\Z/2\Z)\cong \langle t,x,v\mid t^x=t^3,\ v^2=1,\ t^v=t^{-1},\ x^v=x\rangle,
\end{equation}
where $t=ac^{-1}$ and $x=aca^{-1}$ generate the Baumslag-Solitar group $BS(1,3)$ and $v=ab^{-1}$ acts on $BS(1,3)$ by inverting $t$ and leaving $x$ fixed.

Moreover, $G$ has the following finite presentation
\begin{equation}
\label{eqn:presentation_abc}
G\cong\langle a,b,c\ |\ [a,c]=(c^{-1}a)^2,\ (ab^{-1})^2=1,\  (ca^{-1})^{b}=c^{-1}a,\ [c,b^{-1}a]=1\rangle.
\end{equation}
\end{theorem}

\begin{proof}
First, note that elements $t,x$ and $v$ form another generating set for $G$ since we can express the original generators as $a=xt, b=v^{-1}xt$ and $c=x^t$. We define $G$ and generators $t$, $x$ and $v$ in \verb"AutomGrp" package by

\begin{verbatim}
gap> G:=AutomatonGroup("a=(b,b)(1,2),b=(c,a)(1,2),c=(c,a)");
< a, b, c >
gap> t:=a*c^-1;; x:=a*c*a^-1;; v:=a*b^-1;;
\end{verbatim}

Since the relation $t^x=t^3$ is satisfied as shown below, the group $B=\langle t,x\rangle$ is a homomorphic image of $BS(1,3)=\langle \alpha,\beta\mid\alpha^{\beta}=\alpha^3\rangle$. On the other hand, in each proper homomorphic image of $BS(1,3)$ at least one of the images of $\alpha$ and $\beta$ must have a finite order. Since both $t$ and $x$ have infinite order ($t$ and the section $x^2|_{000}=c^2ac^{-1}$ of $x$ at vertex $000$ act transitively on the levels of the tree as can be computed using the algorithm described, for example, in Lemma~2 of~\cite{bondarenko_gkmnss:full_clas32_short} and implemented in \verb"AutomGrp" package), we get that $B\cong BS(1,3)$.

\begin{verbatim}
gap> t^x=t^3;
true
gap> Order(t);
infinity
gap> Order(x);
infinity
\end{verbatim}

The fact that $B$ is normal in $G$ follows from the identities
\begin{equation}
\label{eqn:action_by_v}
\begin{array}{l}
t^v=t^{-1},\\x^v=x.
\end{array}
\end{equation}

\begin{verbatim}
gap> t^v=t^-1;
true
gap> x^v=x;
true
\end{verbatim}

Further, since $v^2=1$, the equalities~\eqref{eqn:action_by_v} immediately imply that $G=B\rtimes\langle v\rangle\cong BS(1,3)\rtimes(\Z/2\Z)$. Finally, by a sequence of Tietze transformations one can convert presentation~\eqref{eqn:presentation_txv} into a presentation~\eqref{eqn:presentation_abc}.
\end{proof}

Similarly to Proposition~\ref{prop:auto2193} we obtain the following:

\begin{proposition}
\label{prop:auto}
The automorphism $\eta$ of a free group $F(a,b,c)$ defined by $\eta(a)=c$, $\eta(b)=b$, $\eta(c)=a$ induces an automorphism of $G$.
\end{proposition}

\begin{proof}
First we verify that the images of relators in the presentation~\eqref{eqn:presentation_abc} of $G$ under $\eta$ are again relators in $G$.
For the relator $r_1=[a,c](c^-1a)^{-2}$ we have
\[\eta(r_1)=[c,a](a^{-1}c)^{-2}=[a,c]^{-1}(a^{-1}c)^{-2}=(c^{-1}a)^{-2}(a^{-1}c)^{-2}=1.\]
For $r_2=(ab^{-1})^2$ we compute
\[\eta(r_2)=(cb^{-1})^2=(t^{-1}v)^2=t^{-1}\left(t^{-1}\right)^v=t^{-1}t=1.\]
For $r_3=(ca^{-1})^{b}(c^{-1}a)^{-1}$ we obtain
\[\eta(r_3)=(ac^{-1})^b(a^{-1}c)^{-1}=\left((ca^{-1})^b\right)^{-1}(a^{-1}c)^{-1}=\left(c^{-1}a\right)^{-1}(a^{-1}c)^{-1}=1.\]
Finally, for $r_4=[c,b^{-1}a]$:
\begin{multline*}
\eta(r_4)=[a,b^{-1}c]=a^{-1}c^{-1}bab^{-1}c=(t^{-1}x^{-1})^2tvxtxtt^{-1}x^{-1}vt^{-1}xt\\
=(t^{-1}x^{-1})^2tx^vt^vt^{-1}xt=(t^{-1}x^{-1})^2txt^{-1}t^{-1}xt\\=t^{-1}x^{-1}t^{-1}t^xt^{-2}xt=t^{-1}x^{-1}t^{-1}t^3t^{-2}xt=1.
\end{multline*}
Therefore, $\eta$ induces an endomorphism of $G$. This endomorphism is obviously onto and also one-to-one since $\eta$ is an involution.
\end{proof}

\begin{corollary}
\label{cor:essfree2372}
The group $G$ acts essentially freely on the boundary of the tree.
\end{corollary}

\begin{proof}
The stabilizer of the first level in $G$ is generated by
\[\begin{array}{lcl}
b^{-1}c^2b^{-1}c&=&(a,c),\\
ab^{-1}c&=&(b,b),\\
c&=&(c,a).
\end{array}\]

As in Corollary~\ref{cor:essfree2193}, in this situation Proposition~\ref{prop:auto} guarantees that we can apply Proposition~\ref{prop:ess_free_aut} and deduce that the action of $G$ on the boundary of the tree is essentially free.
\end{proof}

This subsection treated the last case in the proof of Theorem~\ref{thm:main}, thus finalizing the list of groups acting essentially freely on $\partial T_2$.
\vspace{3mm}

\subsection{Scale-invariant groups.}
We finish the proof of Theorem~\ref{thm:main} with a corollary describing all scale-invariant groups among those that are listed in this theorem.

\begin{corollary}
All groups listed in Theorem~\ref{thm:main} except finite nontrivial groups, $F_3$, and $(\Z/2\Z)*(\Z/2\Z)*(\Z/2\Z)$ are scale-invariant.
\end{corollary}

\begin{proof}
The trivial group acts essentially freely by the definition. Using \verb"AutomGrp" package we deduce that for each group $G$ listed in the statement of this corollary, there is a 3-state automaton generating a group isomorphic to $G$ that is self-replicating.  Therefore, by Proposition~\ref{prop:stab_finite_index}, all these groups are scale-invariant.

On the other hand, $F_3$ is not scale-invariant since finite-index subgroups of $F_3$ are free groups of different ranks that cannot be isomorphic to $F_3$. For $G=(\Z/2\Z)*(\Z/2\Z)*(\Z/2\Z)=\langle a\rangle*\langle b\rangle*\langle c\rangle$ we appeal to the fact that this group has a finite homological type and has a subgroup $H=\langle ab,bc\rangle$ of index 2 isomorphic to a free group $F_2$ of rank 2. Therefore, the virtual Euler characteristic $\chi(G)$ of $G$ is equal to $\frac{\chi(H)}{[G:H]}=\frac{\chi(F_2)}{2}=-\frac12\neq 0$. Thus, for each proper finite index subgroup $K$ of $G$ we have $\chi(K)=[G:K]\chi(G)=-\frac12[G:K]\neq-\frac12=\chi(G)$. This shows that none of the proper finite index subgroups of $G$ is isomorphic to $G$ and hence $G$ is not scale-invariant.
\end{proof}

It is interesting to observe that all scale-invariant groups in the class under consideration are either virtually abelian, or are related to the lamplighter type groups or to the Baumslag-Solitar metabelian groups $B(1,n)$. This gives an additional motivation for Question~\ref{que:scaleinvar} below.

\section{Concluding remarks.}
\label{sec:concluding}
We end our paper with a list of some open questions and concluding remarks.

\begin{question}
Is there a group generated by finite automaton that acts neither essentially freely, nor totally non-freely on the boundary of a rooted tree? (Recall that the action is totally non-free if stabilizers of different points of the set of full measure are different).
\end{question}

\begin{question}
Does the total non-freeness of an action of a group generated by finite automaton  on $\partial T$ imply weak branchness? Observe, that the converse is true~\cite{bartholdi_g:parabolic,grigorch:dynamics11eng}.
\end{question}

\begin{question}
Classify all $(4,2)$-groups and $(2,3)$-groups that act essentially freely on the boundaries of corresponding rooted trees.
\end{question}

\begin{question}
\label{que:scaleinvar}
Are there groups generated by finite automata acting essentially freely on the boundary of rooted tree that are scale-invariant groups and are not based on the use of lamplighter type groups, metabelian Baumslag-Solitar groups $BS(1,n)$, and groups based on constructions from~\cite{nekrashevych_p:scale_invariant}? (See Corollary~\ref{prop:stab_finite_index} for motivation).
\end{question}

\begin{question}
\label{que:heredji}
Is there a hereditary just-infinite group generated by finite automaton? (See Proposition~\ref{prop:heredji} for motivation). Note that any such group will also be an answer to Question~\ref{que:scaleinvar}.
\end{question}

We finish this section with a discussion about singular points of actions that play an important role in the study of topological group actions~\cite{gns00:automata,vorobets:schreier_of_grigorchuk12,grigorch:dynamics11eng,savchuk:thompson2}. Recall that for an action of a group $G$ on a topological space $X$, a \emph{singular point} is a point which is not regular, i.e. such point $x\in X$ that $\Stg(x) \neq \Stg^0(x)$, where $\Stg^0(x)$ denotes the neighborhood stabilizer of $x$ (consisting of elements acting trivially on some neighborhood of $x$). The importance of these points is based on the fact that correspondence $x\mapsto\Stg(x)$ is continuous at regular points (where a natural topology is used on the space of subgroups), while it can be discontinuous at singular points as is observed in~\cite{vorobets:schreier_of_grigorchuk12} and~\cite{savchuk:thompson2}. For essentially free actions with invariant measure whose support is the whole space $X$, the neighborhood stabilizer is trivial for every $x \in X$, so singular points are points with nontrivial stabilizer. In the examples related to actions of self-similar groups on the boundary of rooted tree usually it is not easy to determine all singular points. For instance, for the action of the lamplighter group given by 2 state automaton as in~\cite{grigorch_z:lamplighter} a part of singular points was described in~\cite{nekrashevych_p:scale_invariant}, while the full description is given in~\cite{grigorch_k:lamplighter}. It is strange, that in all known essentially free actions of not virtually abelian groups generated by finite automata there is at least one singular point. It is an interesting open question if this is always the case.


\newcommand{\etalchar}[1]{$^{#1}$}
\def\cprime{$'$} \def\cprime{$'$} \def\cprime{$'$} \def\cprime{$'$}
  \def\cprime{$'$} \def\cprime{$'$} \def\cprime{$'$} \def\cprime{$'$}
  \def\cprime{$'$} \def\cprime{$'$} \def\cprime{$'$} \def\cprime{$'$}

\end{document}